\documentclass[11pt,a4paper]{amsart}

\usepackage{amssymb, amsmath, amsthm, amscd}

\usepackage{enumerate}

\usepackage{hyperref}
\hypersetup{colorlinks = true, linkcolor=red, citecolor=blue}

\theoremstyle{plain} \newtheorem{theorem}{Theorem}[section]
\theoremstyle{plain} \newtheorem{corollary}[theorem]{Corollary}
\theoremstyle{plain} \newtheorem{proposition}[theorem]{Proposition}
\theoremstyle{plain}\newtheorem{lemma}[theorem]{Lemma}
\theoremstyle{definition} \newtheorem{definition}[theorem]{Definition}
\theoremstyle{definition}\newtheorem{example}[theorem]{Example}
\theoremstyle{remark}\newtheorem{remark}[theorem]{Remark}
\theoremstyle{definition}\newtheorem{hypothesis}[theorem]{Hypothesis}
\theoremstyle{remark}
\theoremstyle{definition}
\theoremstyle{definition}\newtheorem{problem}[theorem]{Problem}

\newcommand{\Azero}{\text{\rm Aut}^0(X_0)}
\newcommand{\Aun}{\text{\rm Aut}^1(X_0)}

\newcommand{\T}{\mathcal T(X)}
\newcommand{\D}{\text{\rm Diff}^0(X)}

\newcommand{\DlKz}{\text{\rm Diff}^0_l(X,\mathcal K_0)}
\newcommand{\DK}{\text{\rm Diff}^0(X,\mathop\bigsqcup_{\alpha\in A}\mathcal K_\alpha)}
\newcommand{\DKp}{\text{\rm Diff}^+(X,\mathop\bigsqcup_{\alpha\in A}\mathcal K_\alpha)}
\newcommand{\I}{\mathcal I}

\numberwithin{equation}{section}

\def\uniondisjointe{\mathop\bigsqcup}
\def\union{\mathop\bigcup}

\usepackage{color}
\usepackage{graphicx}
\usepackage{epsfig}

\usepackage{float}

\title[Teichm\"uller and Riemann Stacks]{The Teichm\"uller and Riemann Moduli Stacks}
\author{Laurent Meersseman}
\date{\today}
\thanks{This is part of project Marie Curie 271141 DEFFOL. I enjoyed the warmful atmosphere of the CRM at Bellaterra during the preparation of this work. I would like to thank Ernesto Lupercio for explaining me the basics of groupoids and their interest in moduli theory; Alain Genestier for suggesting the construction of section \ref{Teichmueller} and for stimulating discussions about stacks; Allen Hatcher and Daniel Ruberman for answering some questions about the mapping class group of $4$-manifolds; Fabrizio Catanese for discussions about the rigidified hypothesis of section \ref{rigidified}; Serge Cantat for pointing out a result of Liebermann on automorphisms of k\"ahler manifolds; and Etienne Mann for clarifying some aspects of algebraic stacks. Finally many thanks to the referee whose accurate comments gave rise to a greatly improved version of the paper.}


\address{Laurent Meersseman\\ LAREMA\\ Universit\'{e} d'Angers\\ F-49045 Angers Cedex, France\\ laurent.meersseman@univ-angers.fr}

\dedicatory{To Alberto Verjovsky on his 70th birthday.}

\begin{document}
\begin{abstract}
The aim of this paper is to study the structure of the Teichm\"uller and Riemann moduli spaces, viewed as stacks over the category of complex analytic spaces, for higher-dimensional manifolds. We show that both stacks are analytic in the sense that they admit a smooth analytic groupoid as atlas. We then show how to construct explicitly such an atlas as a sort of generalized holonomy groupoid for such a structure. This is achieved under the sole condition that the dimension of the automorphism group of each structure is bounded by a fixed integer. All this can be seen as an answer to Question 1.8 of \cite{VerbitskySurvey}.
\end{abstract}

\maketitle

\section{Introduction.}
\label{intro}
Let $X$ be a smooth oriented compact surface. The {\it Teichm\"uller space} $\mathcal T(X)$ is defined as the quotient space of the set of smooth integrable complex operators compatible with the orientation (o.c.)
$$
\mathcal I=\{J\ :\ TX\longrightarrow TX\mid J^2\equiv -Id,\ J\text{ o.c.}\}
$$
by $\text{Diff}^0(X)$, the connected component of the identity in the oriented diffeomorphism group $\text{Diff}^+(X)$ of $X$. 

The theory of Teichm\"uller spaces is a cornerstone in complex variables and Riemann surfaces. Originated by Riemann himself and followed by the fundamental works of Teichm\"uller, Ahlfors and Bers, it has moreover implications in many branches of mathematics as algebraic geometry, hyperbolic geometry, complex dynamics, discrete groups, ... 

Perhaps the most basic property of $\mathcal T(X)$ is that
it has a natural structure of a complex manifold, making it a global moduli space of complex structures on $X$.

Moreover, the mapping class group of $X$ acts on $\mathcal T(X)$ and the resulting quotient is a complex orbifold. This refined quotient coincides with the quotient of $\mathcal I(X)$ by the full group $\text{Diff}^+(X)$, the so-called Riemann moduli space $\mathcal M(X)$.
\vspace{5pt}\\
\indent Let now $X$ be a smooth oriented compact manifold of even dimension $2n$ strictly greater than $2$.  The Teichm\"uller and Riemann moduli spaces can still be defined, but one now has to add the integrability condition in the definition
\begin{equation}
\label{I}
\mathcal I=\{J\ :\ TX\longrightarrow TX\mid J^2\equiv -Id,\ J\text{ o.c.},\ \ [T^{1,0}, T^{1,0}]\subset T^{1,0}\}
\end{equation}
for
$$
T^{1,0}=\{v-iJv\mid v\in TX\}.
$$
Although the literature about these higher dimensional Teichm\"uller and Riemann moduli spaces is much less developed than that about surfaces, it has grown significantly in the last years and these spaces play an increasing role in Complex Geometry. Catanese's guide to deformations and moduli \cite{Catsurvey} as well as \cite{Catsurvey2} gives some general local properties of $\mathcal T(X)$ and contains many results on the Teichm\"uller space of minimal surfaces of general type. And in the special case of hyperk\"ahler manifolds, the Teichm\"uller space is used by Verbitsky in a prominent way in his proof of a global Torelli Theorem \cite{Verbitsky} and also to showing some important results on these manifolds \cite{VerbitskyErgodic}.  

However, the main difference with the case of surfaces is that $\mathcal T(X)$  and $\mathcal M(X)$ are {\it just topological spaces} and do not have any good geometric structure. Only for special classes such that the class of hyperk\"ahler manifolds, an analytic structure is known on $\mathcal T(X)$, but even in this case, it is not Hausdorff at all points. Perhaps the most dramatic example is given by $X$ being $\mathbb S^2\times\mathbb S^2$. Then $\mathcal M(X)$, {\it as a set}, can be identified with $\mathbb N$, a point $a\in\mathbb N$ corresponding to the Hirzebruch surface $\mathbb F_{2a}$ (and each connected component of $\mathcal T(X)$ can be identified with $\mathbb Z$, $a$ and $-a$ encoding the same surface, see Examples \ref{Hirzebruch} and \ref{Hirzebruchbis}). But, {\it as a topological space}, it is endowed with a non-Hausdorff topology. No two points are separated, as a consequence of the fact that $\mathbb F_{2a}$ can be deformed onto any $\mathbb F_{2b}$ with $b<a$ by an arbitrary small deformation. Equivalently, this comes from the fact that the dimension of the automorphism group of Hirzebruch surfaces jumps. 

Indeed, in presence of this jumping phenomenon, $\mathcal T(X)$  and $\mathcal M(X)$ are not even locally Hausdorff hence not locally isomorphic to an analytic space (cf. Example \ref{Hopfbis}). This explains why the classical approach developed in the fundamental works of Kodaira-Spencer and Kuranishi is based on the following principles.
\begin{enumerate}
	\item [(i)] in higher dimension the global point of view must be abandoned for the local point of view;
	\item[(ii)] and the Teichm\"uller space replaced with the Kuranishi space which must be thought of as the best possible approximation in the analytic category for a local moduli space of complex structures.
\end{enumerate} 

\indent Nevertheless, putting on $\T$ and $\mathcal M(X)$ a global analytic structure in some sense is the only way to go beyond the classical local deformation theory. As we cannot expect a structure of analytic space, even a non-Hausdorff one, {\it we have to view these quotient spaces as stacks}. The aim of this paper is to develop this point of view. The question now becomes to showing that, {\it as stacks, $\T$ and $\mathcal M(X)$ are analytic}. This can also be seen as an answer to Question 1.8 of \cite{VerbitskySurvey}. Since we work with arbitrary complex structures and not only with projective ones, we have to work with analytic stacks and not algebraic ones.\vspace{5pt}\\
\indent For surfaces of fixed genus $g>1$, the classical setting coincide with the stack setting. Both stacks $\T$ and $\mathcal M(X)$ are analytic and can be fully recovered from the Teichm\"uller and Riemann moduli spaces. In particular, the complex structure on $\T$, respectively the complex orbifold structure on $\mathcal M(X)$ are equivalent to the analytic structures on the corresponding stack. The case of genus $1$ is somewhat more complicated, because of the translations\footnote{To avoid this problem, it is customary to use marked complex tori, that is elliptic curves.}. Here the stack structure contains strictly more information than the classical spaces $\mathbb H$ and $\mathbb H/\text{PSL}_2(\mathbb Z)$ since it also encodes the translation group of each complex torus, but once again both stacks are analytic and their analytic structure comes from the complex structure on the corresponding spaces. 
 \vspace{5pt}\\
\indent  The main results of this paper show that, in any dimension, both the Teichm\"uller and the Riemann moduli stacks are analytic stacks.  The only condition needed for this result to hold is that the dimension of the automorphism group of all structures of $\mathcal T(X)$ (or $\mathcal M(X)$) is bounded by a fixed integer\footnote{which is always the case for surfaces.}. This is nevertheless a mild restriction since we may easily stratify $\mathcal I$ into strata where this dimension is bounded. We emphasize that $X$ can be any compact manifold and that we consider all complex structures and not only projective or k\"ahler ones\footnote{However, our results also apply to the set of k\"ahler structures on $X$ modulo $\text{Diff}^0(X)$ or $\text{Diff}^+(X)$.}.
 \vspace{5pt}\\
\indent  We postpone the precise statements of the main Theorems \ref{maintheorem} and \ref{mainR} as well as the strategy of proof to section \ref{notations} after defining precisely the involved notions. Let us just say that we will follow the same strategy that can be used for Riemann surfaces. Firstly, we define $\mathcal T(X)$ and $\mathcal M(X)$ as stacks of families of complex manifolds diffeomorphic to $X$. This is the easy part. Secondly, we build an atlas with good analytic properties to show they are analytic. This is the difficult part which takes the rest of the paper.
\vspace{5pt}\\
\indent We hope that this paper will serve as a source of motivation for studying global moduli problems in Complex Analytic Geometry and their interplay with analytic stacks. From the one hand, every abstract result on these stacks might apply to moduli problems and increase our knowledge of Complex Manifolds. From the other hand, examples of Teichm\"uller stacks are an unending source of examples of analytic stacks, showing all the complexity and richness of their structure, far from finite dimensional group actions and leaf spaces.

\section{Definitions and statements of the main results}
\label{notations}

\subsection{The Teichm\"uller and Riemann spaces}
\label{TRstacks}
Let $X$ be a smooth (i.e. $C^\infty$) oriented compact connected manifold of even dimension. Let $\mathcal E$, respectively $\mathcal I$, be the space of smooth almost complex, respectively complex operators on $X$ which are compatible with the orientation. The definition of $\mathcal I$ is given in \eqref{I}. We assume that both sets are non-empty.
\vspace{5pt}\\
We topologize $\mathcal E$ as a Fr\'echet manifold locally modelled onto the smooth sections of a vector bundle over $X$ (cf. \cite{Ku3} for the encoding of structures, \cite{Hamilton} and \cite{Verbitsky} for the Fr\'echet topology). 
We denote by $\mathcal E_0$, respectively $\mathcal I_0$, a connected component of $\mathcal E$, respectively $\mathcal I$. Points of $\mathcal E$ will be denoted generically by $J$.
\vspace{5pt}\\
For $T$ a topological space, we denote by $\pi_0(T)$ the set of connected components of $T$.
The previous topology being countable, $\pi_0(\mathcal E)$ is a countable set. 
\vspace{5pt}\\ 
The diffeomorphism group $\text{Diff}^+(X)$ acts on the right on $\mathcal E$ by pullback of almost complex operators. It is a Fr\'echet Lie group \cite{Hamilton} acting analytically\footnote{\label{anaf}There is some subtle point here because the complex structure of $\text{Diff}^+(X)$ depends on the choice of a complex structure on $X$. We will just use the fact that, if we endow locally at identity $\text{Diff}^+(X)$ with chart \eqref{chart}, then the map $(f, J')\mapsto J'\cdot f$ is analytic in a neighborhood of $(Id, J)$.} onto $\mathcal E$. This action leaves $\mathcal I$ invariant. It is given by
\begin{equation}
\label{actionJ}
(J\cdot f)_x (v)=(d_xf)^{-1}\circ J_{f(x)}\circ (d_xf) (v)
\end{equation}
We focus on $\text{Diff}^0(X)$, the connected component of the identity in $\text{Diff}^+(X)$. We define the mapping class group
\begin{equation}
\label{mcgroup}
\mathcal M\mathcal C(X):=\text{Diff}^+(X)\big /\text{Diff}^0(X)
\end{equation}
and we set
\begin{equation}
\label{Teich}
\mathcal T(X):=\mathcal I\big /\text{Diff}^0(X)
\end{equation}
and
\begin{equation}
\label{Riem}
\mathcal M(X):=\mathcal I\big /\text{Diff}^+(X)=\mathcal T(X)\big /\mathcal M\mathcal C(X)
\end{equation}

Both $\mathcal T(X)$ and $\mathcal M(X)$ are endowed with the quotient topology, making them topological spaces.

\begin{remark}
\label{oriented}
In the first version of this paper, we take for $X$ an unoriented smooth compact manifold and consider $\mathcal I$ as the set of all integrable complex operators, regardless of orientation. Then $\mathcal T(X)$ is defined as in \eqref{Teich}, and in \eqref{Riem}, we have to replace the oriented diffeomorphism group by the full diffeomorphism group $\text{Diff}(X)$. This does not change substantially these two sets, and our results apply to this setting. In fact, the main drawback of forgetting the orientation is that the notion of Teichm\"uller space does not coincide to the classical one for surfaces. Especially, the unoriented Teichm\"uller space of a compact surface has two connected components, corresponding to the two possible orientations.
\vspace{5pt}\\
More generally, if $X$ admits a diffeomorphism reversing orientation, then the unoriented Teichm\"uller space has twice more connected components as the classical one. However, the two Riemann spaces coincide. Finally, if $X$ does not admit any orientation reversing diffeomorphism, then the unoriented Teichm\"uller and Riemann spaces are the disjoint union of the classical ones for both orientations. Notice that, in this last case, changing the orientation may completely change the Teichm\"uller and Riemann spaces. It is even possible that they become empty (think of $\mathbb P^2$ and $\overline{\mathbb P^2})$.
\end{remark}

\subsection{Stacks and groupoids.}
\label{stacksgroupoids}
Before getting into the definition of the Teichm\"uller and Riemann moduli stacks, let us define precisely the notions of stacks and groupoids we will use.\vspace{5pt}\\
First, a warning. We insist on the fact that we work exclusively in the $\mathbb C$-analytic context, since we deal with arbitrary compact complex 
manifolds. This forces us to adapt and sometimes to transform the definitions of stacks coming from algebraic geometry. Also, since the literature on stacks over the category of analytic spaces is very scarce, we shall keep the required facts from stack theory to a minimum and give complete proofs even of some routine facts (for example in Proposition \ref{Theyarestacks}).
\vspace{5pt}\\
Moreover, our construction of atlas is inspired in the construction of the \'etale holonomy groupoid of a foliation. So for the groupoid point of view, we stick to the literature in foliation theory and Lie groupoids, especially \cite{Moerdijk}. The conventions are somewhat different from those of algebraic geometry and we have to adapt ourselves to these differences. Especially, we will not make use of the notion of representability (see however Remark \ref{representability}).
\vspace{5pt}\\
Let $\mathfrak S$ denote the category of $\mathbb C$-analytic spaces. We include analytic spaces that are everywhere non-reduced in $\mathfrak S$. We consider it as a site for the euclidean topology: our families of coverings are just standard topological open coverings. We emphasize that we will not use other coverings as \'etale or analytic ones. At some points (for example in Section \ref{examples}), we may restrict the base category to be that of complex manifolds, still with the euclidean coverings.
\vspace{5pt}\\
In this paper, a {\it stack} $\mathcal S$ is a stack in groupoids over the site $\mathfrak{S}$ in the sense of \cite[Def. 8.5.1]{Stacksproject}.
In brief, $\mathcal S$ is a stack if 
\begin{enumerate}
	\item [(i)] $\mathcal S$ comes equipped with a morphism $\mathcal S\to\mathfrak S$ whose fibers are group\-oids.
	\item[(ii)] $\mathcal S\to\mathfrak S$ is a category fibred in group\-oids, i.e. pull-backs exist and are unique up to unique isomorphisms.
	\item[(iii)] Isomorphisms form a sheaf, i.e. one can glue a compatible collection of isomorphisms defined over an open covering of an analytic space $S$ into a single isomorphism over $S\to S$.
	\item[(iv)] Every descent data is effective, i.e. one can glue objects defined on an open covering of an analytic space $S$ into a single object over $S$ by means of a cocycle of morphisms. 
\end{enumerate}

The groupoids we consider are {\it analytic}, that is
\begin{enumerate}
	\item [(i)] the set of objects and the set of morphisms are complex analytic spaces 
	\item[(ii)] and all the structure maps are analytic morphisms.
\end{enumerate}   
\begin{remark}
	\label{nonHausdorff}
	In the setting of Lie groupoids and foliation theory, the space of morphisms is possibly non-Hausdorff, since such phenomena occur when constructing holonomy groupoids of finite-dimensional $C^\infty$ foliations. For example, the holonomy groupoid of the Reeb foliation of the sphere $\mathbb S^3$ is non-Hausdorff. In classical foliation theory, this is linked to the existence of so-called vanishing cycles. Recall also that the Hausdorffness/Non Hausdorffness of the set of morphisms is preserved by Morita equivalence. We refer to \cite[\S 5.2]{Moerdijk} for more details. 
	
	Even if our construction is inspired in that of holonomy groupoids of foliations, all the groupoids we construct will be proved to be Hausdorff. We note that in previous versions of this work, we authorized non-Hausdorff groupoids since at that time we did not succeed in proving our groupoids are Hausdorff.
\end{remark}
Analytic groupoids are in particular topological so that it makes sense to localize them on an open covering of the set of objects \cite{Ha}. The geometric quotient associated to such a groupoid is the topological space obtained by taking the quotient of the set of objects by the equivalence relation defined by the set of morphisms. Connected components of the groupoid refer to connected components of the geometric quotient.\vspace{5pt}\\
Such a groupoid is {\it \'etale}, respectively {\it smooth}, if both source and target maps are \'etale, respectively smooth, morphisms. Here smoothness refers to smoothness of morphisms in analytic/algebraic geometry, not to differentiability. We emphasize that a smooth analytic groupoid is not a complex Lie group\-oid, since we allow singularities of both the set of objects and the set of morphisms, but it is the exact singular counterpart of a complex Lie group\-oid, cf. \cite[\S 5]{Moerdijk}. 
\vspace{5pt}\\
Given a stack, we are above all interested in knowing if it admits an analytic groupoid as atlas (or presentation) in the following sense: the stackification of this groupoid by torsors as explained in \cite{BCEFFGK}, Chapters 3 and 4, is isomorphic to the initial stack. In that case, the stack can be fully recovered from the atlas through this long process of stackification. We will detail this process in the proofs of Theorems \ref{maintheorem} and \ref{mainR}. Among analytic presentations, the following ones are of special interest.

\begin{definition}
	\label{Artinstacks}
	We call a stack {\it \'etale analytic} (respectively {\it Artin analytic} or simply {\it analytic}) if it admits a presentation by an \'etale (respectively smooth) analytic groupoid; {\it Deligne-Mumford analytic} if it is \'etale with finite stabilizers. 
\end{definition}

We take as definition of Morita equivalence that given in \cite[\S 5.4]{Moerdijk}, with the obvious adaptations to the groupoids we use (e.g. replace $C^\infty$ map with $\mathbb C$-analytic map, submersion with smooth morphism, ...). It follows from carefully adapting \cite{BCEFFGK} to the analytic context that two  smooth atlases of the same stack are Morita equivalent.

\begin{remark}
	\label{representability}
	Standard definitions of {\it algebraic stacks} (see for example \cite[Def. 84.12.1]{Stacksproject}) does not involve directly the existence of an atlas but asks for representability of the diagonal and existence of a surjective, smooth morphism from a scheme or an algebraic space onto the stack. It is of course possible to adapt these definitions to the analytic context, but as mentioned in the warning we do not follow this way. However, both notions are not far each from the other. In the algebraic definitions, both conditions are used to ensure the existence of an algebraic atlas. From the surjective and smooth morphism, one constructs a symmetry groupoid which is an atlas for the stack. The set of objects of this groupoid is the scheme or algebraic space on which the smooth morphism is defined. Then, the condition on the diagonal ensures that the smooth morphism is itself representable and so that the set of morphisms of this symmetry groupoid has also a structure of scheme or algebraic space, depending on the precise definition that was taken. So in short, an algebraic stack admits a presentation by an algebraic groupoid. Besides, an analytic stack as defined above surely admits a surjective smooth morphism $f$ from an analytic space $S$ onto it: just take for $S$ the set of objects of the atlas and for $f$ the induced morphism. But we do not know if the diagonal is always representable.
\end{remark}
\subsection{The Teichm\"uller and Riemann moduli stacks as stacks of deformations.}
\label{stacks}
 Let $V$ be an open set of $\mathcal I$. Define the following category $\mathcal M(X,V)$ over $\mathfrak S$. 
\vspace{5pt}\\
{\it Objects} are $(X, V)$-families
\begin{equation}
\label{XVfamily}
\pi\ :\ \mathcal X\longrightarrow B
\end{equation}
that is:
\begin{enumerate}
	\item[(i)] $B\in \mathfrak S$ and $\mathcal X\in\mathfrak S$.
	
	\item[(ii)] $\pi$ is a smooth and proper morphism with reduced fibers all diffeomorphic to $X$.
	
	\item[(iii)] Each fiber $X_b:=\pi^{-1}(b)$ can be encoded as $(X, J)$ with $J\in V$.
\end{enumerate}
In other words, a $(X, V)$-family is nothing else than an analytic deformation of complex structures of $X$ such that the structure of each fiber is isomorphic to a point of $V\subset \mathcal I$. Of course, if $V^{sat}$ denotes the image of $V$ through the action of $\text{Diff}^+(X)$, then $\mathcal M(X,V)$ and $\mathcal M(X,V^{sat})$ are equal. However, it is interesting to have this flexibility, for example we will often take for $V$ a connected component of $\mathcal I$, even if it is not saturated.
\vspace{5pt}\\
{\it Morphisms} are cartesian diagrams 
\begin{equation}
\label{XVmorphisms}
\begin{CD} 
\mathcal X@ >F>>\mathcal X'\cr
@ V\pi VV @ VV \pi' V\cr
B@ >f>>B'
\end{CD}
\end{equation}
between $(X,V)$-families. Observe that the pull-back of a $(X,V)$-family is a $(X,V)$-family.\vspace{5pt}\\
We now pass to the construction of $\mathcal T(X, V)$, which is more delicate. Observe that any family 
$\pi :\mathcal X\to B$ can be seen locally over some sufficiently small open set $B_\alpha\subset B$ as 
\begin{equation}
\label{trivialization}
\mathcal X_{\vert B_\alpha}\simeq (X\times B_\alpha, \mathcal J_\alpha)
\end{equation}
for some smooth family $\mathcal J_\alpha$ of complex operators of $X$. Over an intersection $B_\alpha\cap B_\beta$, two trivializations \eqref{trivialization} are glued using a family $(\phi_t)_{t\in B_\alpha\cap B_\beta}$ of diffeomorphisms of $X$ whose differential commute with $\mathcal J_\alpha$ and $\mathcal J_\beta$. As a consequence, $\mathcal  X$ is diffeomorphically a bundle over $B$ with fiber $X$ and structural group $\text{Diff}^+(X)$. In particular, once such an identification with a bundle is fixed, it makes sense to speak of the structural group of $\mathcal X$, and to make a reduction of the structural group to some subgroup $H$ of $\text{Diff}^+(X)$. And it makes also sense to speak of $H$-isomorphism of the family $\mathcal X$, that is isomorphism of $\mathcal X$ such that, in each fiber, the induced diffeomorphism of $X$ is in $H$.
\vspace{5pt}\\
We define $\mathcal T(X,V)$ as the category whose objects are $(X,V)$-families with a marking
\begin{equation}
\label{XVfamilymarked}
\begin{CD}
\mathcal X @>\simeq >> E\\
@V\pi VV @VVV\\
B @= B
\end{CD}
\end{equation}
with $E\to B$ a bundle with fiber $X$ and structural group reduced to $\text{Diff}^0(X)$ and whose morphisms are cartesian diagrams \eqref{XVmorphisms} such that the canonical isomorphism between $\mathcal X$ and $f^*\mathcal X'$ induces a $\text{Diff}^0(X)$-isomorphism of the markings.
\vspace{5pt}\\
Alternatively, one may use $\D$-framings, that is $C^\infty$-isotopy classes of maps 
\begin{equation}
\begin{CD}
X &@>\simeq >i> &\pi^{-1}(b)\ &\hookrightarrow &\mathcal X\\
&&&&@V\pi VV @VV\pi V\\
&&&&b&\hookrightarrow &B
\end{CD}
\end{equation}
Here $b$ is any point of $B$ and isotopies are $C^\infty$-maps $I$ from $X\times [0,1]$ to $\mathcal X$ such that $\pi\circ I(X\times\{t\})$ is a point for all $t\in [0,1]$. In particular, we may replace $b$ by any other point using an isotopy. Set $I_t:=I(-,t)$. Since $\mathcal X\to B$ is diffeomorphic to a $\D$-bundle, then, given an isotopy $I$ with $\pi\circ I(X\times\{0\})=\pi\circ I(X\times\{1\})$, the diffeomorphism $I_1^{-1}\circ I_0$ of $X$ belongs to $\D$. In other words, the framing induces in that case a coherent identification of the fibers with $X$ up to an element of $\D$.   \vspace{5pt}\\
This forms a category over $\mathfrak S$ and a subcategory of $\mathcal M(X,V)$. In general, it contains stricly less objects, since some $(X,V)$-families do not admit $C^\infty$-markings. 
\begin{proposition}
	\label{Theyarestacks}
	Let $V$ be an open subset of $\mathcal I$. Then both $\mathcal M(X,V)$ and $\mathcal T(X,V)$ are stacks.
\end{proposition}

\begin{proof}
It is straightforward but we sketch it for sake of completeness. First, the natural morphism $\mathcal M(X,V)\to\mathfrak S$ is obviously a category fibred in groupoids. The fiber over $S\in\mathfrak S$ is the groupoid formed by $(X,V)$-families over $S$ as objects and isomorphisms of families as morphisms. Then given two $(X,V)$-families $\pi : \mathcal X\to S$ and $\pi' : \mathcal X'\to S$ and an open covering $(S_\alpha)$ of $S$, any collection of isomorphisms $f_\alpha$ from the restriction of $\mathcal X$ to $\pi^{-1}(S_\alpha)$ onto the restriction of $\mathcal X'$ to $(\pi')^{-1}(S_\alpha)$ such that $f_\alpha$ and $f_\beta$ are equal on the intersections $\pi^{-1}(S_\alpha\cap S_\beta)$ obviously glue to give an isomorphism of families between $\mathcal X$ and $\mathcal X'$. So isomorphisms form a sheaf. Finally, starting from a collection $\pi_\alpha :\mathcal X_\alpha\to S_\alpha$ of $(X,V)$-families and from a cocycle $f_{\alpha\beta}$ of isomorphisms of families between $\pi^{-1}_\alpha (S_\alpha\cap S_\beta)$ and $\pi^{-1}_\beta (S_\alpha\cap S_\beta)$, then
$$
\mathcal X:=\uniondisjointe \mathcal X_\alpha/\equiv
$$
where $\equiv$ is the equivalence relation given by the cocycle $(f_{\alpha\beta})$, is a $(X,V)$-family over $S$. Every descent data is effective. The proof for $\mathcal T(X,V)$ is similar.
\end{proof}
We may thus define
\begin{definition}
	\label{RMS}
	We call {\it Riemann moduli stack} the stack $\mathcal M(X,\mathcal I)$. The stack $\mathcal M(X, V)$ is the Riemann moduli stack for complex structures belonging to $V$. 
\end{definition}

By abuse of notation, we denote simply by $\mathcal M(X)$ the Riemann moduli stack. No confusion should arise with \eqref{Riem}. In the same way, 

\begin{definition}
	\label{TMS}
	We call {\it Teichm\"uller stack} the stack $\mathcal T(X,\mathcal I)$. The stack $\mathcal T(X, V)$ is the Teichm\"uller stack for complex structures belonging to $V$. 
\end{definition}

By abuse of notation, we denote simply by $\mathcal T(X)$ the Teichm\"uller stack. No confusion should arise with \eqref{Teich}.

\subsection{Statement of the main results}
\label{statements}
Let $J\in \mathcal I$ and set 
\begin{equation}
\label{XJ}
X_J:=(X,J)
\end{equation}
\begin{remark}
	\label{notation}
	To avoid cumbersome notations, we write $X_0$ for $X_{J_0}$, and $X_\alpha$ for $X_{J_\alpha}$, ...
\end{remark}
Let $\Theta_J$ be the sheaf of germs of holomorphic vector fields on $X_J$. For $i\geq 0$, we consider the function
\begin{equation}
\label{h(J)}
J\in\mathcal I\longmapsto h^i(J):=\dim H^i(X_J,\Theta_J).
\end{equation}
Set
\begin{equation}
\label{I(k)}
             \mathcal I(k)=\{J\in \mathcal I\mid h^0(J)\leq k\}\quad\text{ and }\quad\I_0(k)=\I_0\cap \I(k)
\end{equation}
The $\mathcal I(k)$ and $\I_0(k)$ are open sets of $\mathcal{I}$, see \eqref{Ic}.
\begin{theorem}
	\label{maintheorem}
	Let $V$ be an open set of $\mathcal I$ (for example, $V$ is a connected component of $\mathcal I$). For all $k$, define $\mathcal I(k)$ as in {\rm \eqref{I(k)}}. Then,
	\begin{enumerate}
		\item [(i)] For all $k$, the stack $\mathcal T(X,V\cap \mathcal I(k))$ is Artin analytic.
		\item[(ii)] Assume that the function $h^0$ is bounded on $V$, resp. on $\I$. Then, the stack $\mathcal T(X,V)$, resp. $\T$, is Artin analytic.
	\end{enumerate}
\end{theorem}
and
\begin{theorem}
	\label{mainR}
	Let $V$ be an open set of $\mathcal I$ (for example, $V$ is a connected component of $\mathcal I$). For all $k$, define $\mathcal I(k)$ as in {\rm \eqref{I(k)}}. Then,
	\begin{enumerate}
		\item [(i)] For all $k$, the stack $\mathcal M(X,V\cap \mathcal I(k))$ is Artin analytic.
		\item[(ii)] Assume that the function $h^0$ is bounded on $V$, resp. on $\I$. Then, the stack $\mathcal M(X,V)$, resp. $\mathcal M(X)$, is Artin analytic.
	\end{enumerate}
\end{theorem}
Both Theorems will be proved in Section \ref{main} as easy consequences of the more precise Theorems \ref{maintheorembis} and \ref{mainRbis}.

\subsection{Strategy of proof and organization of the paper.}
We want to construct smooth analytic atlases of $\mathcal T(X,V)$ and $\mathcal M(X,V)$. Since we deal with arbitrary complex structures, one has to use as a starting point the classical deformation theory of Kodaira-Spencer and to build such an atlas from the local data encoded in the Kuranishi space. But, in the higher dimensional case, the jumping phenomenon causes many troubles. When it occurs, a positive-dimensional subspace of the Kuranishi space encode a single complex structure and we have to encode in our atlas that all these points are the same complex structure. And this is not the only problem. One should expect that knowing, at a complex structure $J$, the Kuranishi space of $(X,J)$ and the identifications induced by its automorphisms, respectively by its automorphisms which are $C^\infty$-isotopic to the identity is enough to get a local model of the Riemann moduli stack, respectively of the Teichm\"uller stack.
\vspace{5pt}\\
\indent This is however not correct. A third element is missing. Some orbits of $\text{Diff}^0(X)$ may a priori have a complicated geometry and accumulate onto $J$. This induces additional identifications to be done in the Kuranishi space, and thus to encode in the atlas of the Riemann or Teichm\"uller stack, even in the absence of automorphisms.
\vspace{5pt}\\
\indent The main problem behind this atlas construction is to understand how to glue the bunch of Kuranishi spaces, in other words how to keep track of all identifications to be done not only on a single Kuranishi space but also between different ones.\vspace{5pt}\\
\indent This is achieved here by describing the space of complex structures $\mathcal I$ as a foliated space in a generalized sense. Then, we describe the stacky structure of the leaf space. A natural source of stacks is given by (leaf spaces of) foliations. Such stacks admit atlases given by an \'etale groupoid, the \'etale holonomy groupoid \cite[\S 5.2]{Moerdijk}. In general, the action of $\text{Diff}^0(X)$ onto $\mathcal I$ does not define a foliation, nor a lamination. But we show that it defines a more complicated foliated structure. We then turn to the construction of an associated holonomy groupoid. It is however much more involved than the classical construction and it constitutes the bulk of the paper. Indeed, the transverse structure of this generalized foliation being stacks, the holonomy morphisms are stacks morphisms and do not fit easily into a nice groupoid. A lot of work is needed for that. 
\vspace{5pt}\\
\indent The paper is organized as follows. Recall that we defined the Teichm\"uller and the Riemann moduli stacks and stated precisely the main results in section \ref{notations}. We collect some facts about the Kuranishi space in Section \ref{Kuranishisection} and we explain how to turn it into a Kuranishi stack that encodes also the identifications induced by the automorphisms. We then give some general properties of $\mathcal I$ in section \ref{f-homotopy}, putting emphasis on connectedness properties, and introducing a graph, called the {\it graph of $f$-homotopy}. 
The foliated structure of $\mathcal I$ is introduced in section \ref{TGfoliationsection}. The technical core of the paper is constituted by sections \ref{multifoliate} and \ref{Teichmueller}, where we perform the construction of the analogue for the holonomy groupoid. We call it {\it the Teichm\"uller groupoid}. To smoothe the difficulties of the construction, a sketch of it is given in section \ref{holonomy} and a very simple case is treated in section \ref{rigidified}. All this culminates in the proof of Theorem \ref{maintheorembis}, stating that the Teichm\"uller groupoid is an analytic smooth presentation of the Teichm\"uller stack. Analogous construction and statement for the Riemann moduli stack are done in sections \ref{Riemann} and \ref{main}. Complete examples are given in section \ref{examples}.

\section{The Kuranishi stack.}
\label{Kuranishisection}

\subsection{The Kuranishi space and Theorem.}
\label{Kuranishi}
Fix a riemannian metric on $X$ and let $\exp$ denote the exponential associated to this metric. 

 For any $J\in\I$, a complex chart for $\text{Diff}^0(X)$ at $Id$ is given by the map
\begin{equation}
\label{chart}
e\ :\ \xi\in W\subset A^0\longmapsto \exp (\xi+\bar \xi)\in \text{Diff}^0 (X)
\end{equation}
where $A^0$ is the $\mathbb C$-vector space of $(1,0)$-vector fields of $X_J$ and $W$ a neighborhood of $0$.
\vspace{5pt}\\
We denote by $\text{Aut}(X_J)$ the group of automorphisms of $X_J$. The connected component of the identity $\text{Aut}^0(X_J)$ in $\text{Aut}(X_J)$ is tangent to $H^0(X_J,\Theta_J)$. We define
\begin{equation}
\label{Aut1}
\text{Aut}^1(X_J):=\text{Aut}(X_J)\cap\text{Diff}^0(X).
\end{equation}
\begin{remark}
\label{Autattention}
Be careful that \eqref{Aut1} {\it is not} equal to $\text{Aut}^0(X_J)$, cf. section \ref{rigidified} and \cite{Aut1paper}.
\end{remark}

Let $J_0\in \mathcal I$. Kuranishi's Theorem \cite{Ku1}, \cite{Kur1}, \cite{Ku3} gives a finite dimensional local model for $\mathcal I$ and the action of $\text{Diff}^0(X)$, namely

\begin{theorem}{\bf (Kuranishi, 1962)}.
\label{KuranishiTheorem}
For any choice of a closed complex vector space $L_0$ such that
\begin{equation}
\label{A0}
A^0=L_0\oplus H^0(X_0,\Theta_{0})
\end{equation}
there exists a connected open neighborhood $U_0$ of $J_0$ in $\mathcal I$, a finite-dimensional analytic subspace $K_0$ of $U$ containing $J_0$ and an analytic isomorphism (onto its image)
\begin{equation}
\label{Phi0}
\Phi_0\ :\ U_0\longrightarrow K_0\times L_0
\end{equation}
such that 
\begin{enumerate}
\item[(i)] The inverse map is given by
\begin{equation}
\label{Kuranishiaction}
(J,v)\in \Phi_0(U_0)\longmapsto J\cdot e(v).
\end{equation}
\item[(ii)] The composition of the maps
\begin{equation}
\label{id}
\begin{CD}
K_0\hookrightarrow U_0@> \Phi_0 >> K_0\times L_0 @>\text{1st projection} >>K_0
\end{CD}
\end{equation}
is the identity.
\end{enumerate}
\end{theorem}

\begin{remark}
Indeed, Kuranishi always uses the $L^2$-orthogonal complement of the space $H^0(X_0,\Theta_{0})$ as $L_0$. However, it is easy to see that everything works with any other closed complement, cf. \cite{circle}.
\end{remark}

\begin{remark}
\label{Sobolev}
Theorem \ref{KuranishiTheorem} is proved using the inverse function Theorem. To do that, one extends $\mathcal E$ to operators of Sobolev class $L^2_l$ (with $l$ big), so that $\mathcal E$ becomes a Hilbert manifold. Then one may use the classical inverse function Theorem for Banach spaces to obtain the isomorphism \eqref{Phi0}. Finally, because $K_0$ is tangent to the kernel of a strongly elliptic differential operator, then it only consists of $C^\infty$ operators and the isomorphism \eqref{Phi0} is still valid when restricting to $C^\infty$ operators, see \cite{Douady}, \cite{Kur1} and \cite{Ku3} for more details.
\end{remark}

Following \cite{circle}, we call such a pair $(U_0,L_0)$ a {\it Kuranishi domain} based at $J_0$. We make the following assumption
\begin{hypothesis}
	\label{squarehyp}
	The image of $\Phi_0$ is contained in a product $K_0\times W_0$ with $W_0\subset W\cap L_0$ an open and connected neighborhood of $0$ in $L_0$.
\end{hypothesis}
Moreover, we call $\Xi_0$ the natural retraction map
\begin{equation}
\label{Xi0}
\begin{CD}
\Xi_0\ :\ U_0@>\Phi_0 >>K_0\times W_0 @> \text{1st projection} >> K_0
\end{CD}
\end{equation}
and $\Upsilon_0$ the other projection
\begin{equation}
\label{Up0}
\begin{CD}
\Upsilon_0\ :\ U_0@>\Phi_0 >>K_0\times W_0 @> \text{2nd projection} >> W_0
\end{CD}
\end{equation}
Given $J\in\mathcal I$, we denote by $K_J$ the Kuranishi space of $X_J$. We use the same convention for $K$ as that stated for $X$ in Remark \ref{notation}. 

\begin{remark}
 \label{Kuniqueness}
 It is a classical fact that the {\it germ} of $K_J$ at $J$ is unique up to isomorphism. However, in this paper, we consider $K_J$ as an analytic subspace of $\mathcal I$, not as a germ. By abuse of terminology, we nevertheless speak of {\it the} Kuranishi space.
\end{remark}
\subsection{Automorphisms and the Kuranishi stack}
\label{localaction}
The complex Lie group $\text{Aut}^1(X_0)$ (respectively $\text{Aut}(X_0)$) is the isotropy group for the action of $\text{Diff}^0(X)$ at $J_0$ (respectively $\text{Diff}^+(X)$). We focus on the connected component of the identity $\text{Aut}^0(X_0)$ in this isotropy group. It acts on $\mathcal I$, and so locally on $U_0$. This action induces a local action of each $1$-parameter subgroup on $K_0$. In other words, let now $f$ be an element of $\Azero$. There exists some maximal open set $U_f\subset K_0$ such that 
\begin{equation}
\label{Aut0action}
\text{Hol}_f\ :\ J\in U_f\subset K_0\longmapsto Jf:=\Xi_0(J\cdot f)\in K_0
\end{equation}
is a well defined analytic map. Observe that $\text{Hol}_f$ fixes $J_0$. We want to encode all these maps \eqref{Aut0action} in an analytic groupoid
\begin{equation}
 \label{tgdef}
 \mathcal A_0\rightrightarrows K_0.
\end{equation}
\begin{remark}
 \label{notanaction}
  Although it is the case in many examples, the groupoid \eqref{tgdef} will not in general describe a local $G$-action. This comes from the fact that there is no reason for $J(g\circ h)$ to equal $(Jg)h$. In particular, there is no reason for the isotropy groups of the groupoid to be subgroups of $ \text{Aut}^0(X_0)$. They are just submanifolds. Hence we will need some work to define it precisely.
\end{remark}
We start with the following Lemma. We recall that $W_0$ is the neighborhood of $0$ in $L_0$ appearing in Hypothesis \ref{squarehyp}.
\begin{lemma}
	\label{Gdecompolemma} We have
	\begin{enumerate}
		\item [(i)] If $W_0$ is small enough, then there exist an open and connected neighborhood $T_0$ of the identity in $\text{\rm Aut}^0(X_0)$ and an open and connected neighborhood $D_0$ of the identity in $\text{\rm Diff}^0(X)$ such that  
		\begin{equation}
		\label{Gdecompo}
		(\xi,g)\in W_0\times T_0\longmapsto g\circ e(\xi)\in D_0
		\end{equation}
		is an isomorphism. 
		\item[(ii)] Set $\mathcal D_0=\bigcup_{g\in \text{\rm Aut}^0(X_0)}gD_0$. Then \eqref{Gdecompo} extends as an isomorphism
		\begin{equation}
		\label{Gdecompoglobal}
		(\xi,g)\in W_0\times \text{\rm Aut}^0(X_0)\longmapsto g\circ e(\xi)\in \mathcal D_0
		\end{equation}
	\end{enumerate}
	
\end{lemma}

\begin{proof}
	Pass to vector fields and diffeomorphisms of Sobolev class $L^2_l$ for some big $l$ and extend the map. Since $T$ consists of holomorphic elements, this map is of class $C^\infty$ and a simple computation shows that its differential at $(0, Id)$ is an isomorphism. Hence we may apply the local inverse Theorem and get the result for this Sobolev class. To finish with point (i), it is enough to remark that, since $g$ is holomorphic, $g\circ e(\xi)$ is of class $C^\infty$ if and only $\xi$ is.
	
	This also proves that \eqref{Gdecompoglobal} is a local isomorphism at each point. Indeed, for $g_0\in \Azero$, the map $(\xi, g)\in W_0\times g_0T_0\mapsto g_0g\circ e(\xi)\in g_0D_0$ is an isomorphism by point (i). Since it is clearly surjective, we just have to check injectivity. Assume that
	$$
	g\circ e(\xi)=g'\circ e(\xi')\qquad\text{ with }\xi\in W_0,\ \xi'\in W_0, g\in  \text{\rm Aut}^0(X_0), g'\in \text{\rm Aut}^0(X_0).
	$$
	Making this diffeomorphism act on $J_0$, we obtain
	\begin{equation}
	\label{Gdecompuni}
	J_0\cdot e(\xi)=J_0\cdot (g\circ e(\xi))=J_0\cdot (g\circ e(\xi'))=J_0\cdot e(\xi')
	\end{equation}
	hence applying $\Upsilon_0$ to \eqref{Gdecompuni} yields $\xi=\xi'$ and thus $g=g'$.
\end{proof}
\begin{remark}
	\label{noncommutatif}
	Note the order in \eqref{Gdecompo}. If we consider the map $(\xi,g)\mapsto e(\xi)\circ g$, the above proof does not apply. Indeed, this last map is not $C^1$ for vector fields and diffeomorphisms of Sobolev class $L^2_l$, cf. \cite[Example I.4.4.5]{Hamilton}. 
\end{remark}
We say that $(J,F)$ is {\it $(U_0,\mathcal D_0)$-admissible} if $J$ belongs to $K_0$ and $F$ is a finite composition of diffeomorphisms $F_1,\hdots, F_k$ of $\mathcal D_0$ such that $J_{i+1}:=J_i\cdot F_i$ belongs to $U_0$ for $i$ between $1$ and $k$ and with the convention $J_1:=J$. 

In particular, we have $J\cdot F\in U_0$, so replacing $F$ with $F\circ e(\Upsilon_0(J\cdot F))$ if necessary, we obtain a new $(U_0,\mathcal D_0)$-admissible couple such that $J\cdot F$ belongs to $K_0$. In the same way, replacing $F_1$ with $F_1\circ e(\Upsilon_0(J\cdot F_1))$, then $F_2$ with $(e(\Upsilon_0(J\cdot F_1)))^{-1}\circ F_2\circ e(\Upsilon_0(J\cdot F_2))$ and so on, we may assume that every $J_i$ belongs to $K_0$. In the sequel, we always assume that an $(U_0,\mathcal D_0)$-admissible couple has this property.

Now define 
\begin{equation}
\label{Azero}
\mathcal A_0=\{(J,F)\in \text{Diff}^0 (X,\mathcal K_0)\mid (J,F)\text{ is }(U_0,\mathcal D_0)\text{-admissible}\}
\end{equation}
where $\text{Diff}^0 (X,\mathcal K_0)$ denotes the set of $C^\infty$ diffeomorphisms from $X$ to a fiber of the Kuranishi family $\mathcal K_0\to K_0$. Here by $(J,F)\in \text{Diff}^0 (X,\mathcal K_0)$, we mean that we consider $F$ as a diffeomorphism from $X$ to the complex manifold $X_J$. We also consider the two maps from $\mathcal U_0$ to $K_0$
\begin{equation}
\label{st}
\alpha (J,F)=J\qquad\text{ and }\qquad \beta (J,F)=J\cdot F
\end{equation}
\begin{remark}
	\label{important}
	There is a subtle point here. In order to define \eqref{tgdef} as a smooth analytic groupoid, we will realize $\mathcal A_0$ as an analytic subspace of $\text{Diff}^0 (X,\mathcal K_0)$. We emphasize that the complex structure on $\text{Diff}^0(X,\mathcal K_0)\simeq \D\times K_0$ is not a product structure. Indeed, the $L^2_l$-completion of $\D$ can be endowed with a complex structure as an open set of the complex Banach manifold of $L^2_l$-maps from $X$ to $X_0$, but this complex structure depends on $X_0$, that is depends on the choice of a complex structure on $X$. If we cover it with \eqref{chart} as complex chart at identity and with complex chart $F\circ e$ at $F$, then the changes of charts depend on $A^0$ and thus on $J_0$. We set $L^2_l(X,X_0)$, resp. $\text{Diff}^0_l (X,X_0)$ for this Banach manifold and more generally $L^2_l(X,X_J)$, resp. $\text{Diff}^0_l (X,X_J)$. Now the completion $L^2_l(X,\mathcal K_0)$ (and thus the completion $\text{Diff}^0_l (X,\mathcal K_0)$ of $\text{Diff}^0 (X,\mathcal K_0)$ as open subset of $L^2_l(X,\mathcal K_0)$) can be endowed with a structure of a complex Banach analytic space such that the natural projection onto $K_0$ is smooth with fiber over $J$ equal to $L^2_l(X,X_J)$, see \cite{DouadyBourbaki}. As a consequence, we will show in the proof of Lemma \ref{analyticKurstack} that $\mathcal A_0$ is locally modelled onto $K_0\times\Azero$ but is not realized in general as an open submanifold of it (cf. Remark \ref{notanaction}).
	For example, if $X_0$ is an elliptic curve $\mathbb E_\tau$ and $K_0$ is a neighborhood of $\tau$ in the upper half-plane $\mathbb H$, then $\mathcal A_0$ is {\it diffeomorphic} to $\Azero\times K_0$ that is to $\mathbb E_\tau\times K_0$ but, as a complex manifold, $\mathcal A_0$ is in fact the universal family over $K_0$, that is the family whose fiber over $\tau'\in\mathbb H$ is $\mathbb E_{\tau'}$ (cf. \cite{Namba}). 
\end{remark}
From remark \ref{important} and the proof of Lemma \ref{Gdecompolemma}, we have indeed
\begin{lemma}
	\label{Gdecompolemmaanalytic}
	Let $W_0^l$ be a connected neighborhood of $0$ in the $L^2_l$-completion of $L_0$ such that $W_0=W_0^l\cap L_0$. Then, if $W_0^l$ is small enough, the map {\rm \eqref{Gdecompo}}, resp. {\rm \eqref{Gdecompoglobal}}, extends as an analytic isomorphism from $W_0^l\times T_0$, resp. $W_0^l\times \Azero$, to $\text{\rm Diff}^0_l (X,X_0)$. 
\end{lemma}
Then we set
\begin{equation}
\label{multgroupoid}
m((J,F),(J\cdot F, F'))=(J,F\circ F'),\qquad i(J,F)=(J\cdot F, F^{-1})
\end{equation}
and $n(J)=(J,Id)$.
We have
\begin{proposition}
	\label{analyticKurstack}
	The groupoid $\mathcal A_0\rightrightarrows K_0$ endowed with structure maps described in {\rm \eqref{st}} and {\rm \eqref{multgroupoid}} is a smooth analytic groupoid.
\end{proposition}
\begin{proof}
	The space $\mathcal A_0$ is an analytic subspace of $\DlKz$ as an open subset of the set of $(J,F)$ in $\DlKz$ such that $J\cdot F$ satisfies the analytic equations defining $K_0$ as an analytic subspace of (the completion of) $\I$ and thus of (the completion of) $\mathcal E$. Also, $\alpha$ is just the restriction to this analytic subspace of the projection of $L^2_l(X,\mathcal K_0)$ onto $K_0$, hence is analytic. And $\beta$ is given by the action $(J,F)\mapsto J\cdot F$ hence is also analytic.
	
	Let now $(J,F)$ belong to $\mathcal A_0$. Let $A$ be a neighborhood of $(J,F)$ in $\mathcal A_0$ such that $F^{-1}\circ F'$ belongs to $D_0$ for all points $(J',F')$ of $A$. Consider the following composition of analytic maps
	\begin{equation}
	\label{KuranishiStackChart}
	\begin{split}
	(J',F')\in A\longmapsto (J',\chi)\in &K_0\times W^l\\\longmapsto &(J',f',\xi')\in K_0\times \Azero\times W_0^l\\&\longmapsto (J',f')\in K_0\times\Azero
	\end{split}
	\end{equation}
	The first one is the restriction of the inverse of the chart $F\circ e$ to $A$, hence satisfies $F'=F\circ e(\chi)$; the second one is given by Lemmas \ref{Gdecompolemma} and \ref{Gdecompolemmaanalytic}, hence $e(\chi)=f'\circ e(\xi')$; and the third one is just the projection. The first two maps are obviously analytic isomorphisms onto their image. For the third one, its inverse is given by the formula
	\begin{equation}
	\label{inverseKuranishiStackChart}
	(J',f')\longmapsto (J',f',\Upsilon_0(J'\cdot (F\circ f')))
	\end{equation}
	We note that the composition in \eqref{KuranishiStackChart} is independent of $l$ and that $\mathcal A_0$ is locally modelled on the product of a neighborhood of a point in $K_0$ with some open neighborhood of the identity in $\Azero$. 
	
	Moreover, it shows that $\alpha$ is a smooth morphism, since it is given by the projection map
	\begin{equation}
	\label{alphasmooth}
	(J',f')\in C\subset K_0\times\Azero\longmapsto J'\in K_0
	\end{equation}
	in the chart given by \eqref{KuranishiStackChart} ($C$ is the image of this chart). This also shows that the anchor map $n$ is analytic as it is locally given by the section $J\mapsto (J,Id)$ to \eqref{alphasmooth} for $F=Id$. 
	
	Now, observe that $L^2_l(X,X_J)$ and $L^2_l(X,X_{J\cdot F})$ are isomorphic Banach manifolds. Moreover, composition in $L^2_l(X,X_J)$ is not analytic (cf. Remark \ref{noncommutatif}) but it is when restricted to finite-dimensional complex submanifolds/subspaces containing only $C^\infty$ structures. These two observations show that the multiplication is analytic. In the same way the inverse map of the groupoid is analytic. Finally, since the source map is smooth and the inverse map is analytic, this implies that the target map is also smooth. 
\end{proof}
\begin{definition}
	The {\it Kuranishi stack} associated to $K_0$ is the stackification of {\rm \eqref{tgdef}}.
\end{definition}
We have  
\begin{proposition}
	\label{geometricquotient}
	The geometric quotient of the Kuranishi stack is homeomorphic to the topological space $U_0/\sim$, for $U_0$ defined as in \eqref{Phi0} and $J\sim J'$ is the equivalence relation generated by $J'=J\cdot F$ for $F$ in the neighborhood $\mathcal D_0$ of $\Azero$ in $\D$ of Lemma {\rm \ref{Gdecompolemma}}.
\end{proposition}

This is a direct consequence of definition \eqref{Azero} (compare with \cite{KuranishiNote}). 

\begin{remark}
 \label{localmodeldifference}
 However, the geometric quotient of the Kuranishi stack {\it has no reason to be homeomorphic to} the topological space $U_0/\sim$ for $J\sim J'$ the equivalence relation generated by $J'=J\cdot f$ for $f\in\D$, because there may exist $f$ with $J$ and $J\cdot f$ in $U_0$ but such that $(J,f)$ is not $(U_0,\mathcal D_0)$-admissible. Rephrasing this important remark, the Teichm\"uller stack is not locally isomorphic to the Kuranishi stack, cf. Remark \ref{Catanese}.
\end{remark}

\begin{remark}
	\label{localactionbefore}
	In many cases, the groupoid \eqref{tgdef} is a translation groupoid, although its structure is much more complicated in general. For that reason, in previous versions of this paper, we denote it abusively by $\Azero\times K_0\rightrightarrows K_0$.
\end{remark}

We now want to link the structure of \eqref{tgdef} with the foliated structure of $K_0$ described in \cite{ENS}. Recall that the leaf through a point $J_1$ is the maximal connected subset of $K_0$ all of whose points encode $J_1$ up to isotopy. We have
\begin{proposition}
	\label{autfol}
	The space of connected components of the classes of $\sim$ in $U_0$ is homeomorphic to the leaf space of $K_0$ by its foliated structure.
\end{proposition}
\begin{proof}
	Let $J_2$ be in the leaf through $J_1$. Then there exists an isotopy $(f_t)$ such that  
	\begin{equation*}
	\text{for all } 0\leq t\leq 1\qquad J_1\cdot f_t\in K_0,\qquad f_0\equiv Id,\qquad J_1\cdot f_1=J_2.
	\end{equation*}
	So $(J_1,f_t)$ is $(U_0,\mathcal D_0)$-admissible for all $t$ and $J_2$ belongs to the connected component containing $J_1$ of the equivalence class of $J_1$. The converse is obvious.	
\end{proof}

\section{Connectedness properties of $\mathcal I$ and the graph of $f$-homotopy.}
\label{f-homotopy}

Observe that Kuranishi's Theorem \ref{KuranishiTheorem} implies that $\mathcal I$ is locally $C^\infty$-pathwise connected in $\mathcal E$. Therefore, 

\begin{proposition}
\label{countable1}
We have:
\begin{enumerate}
\item[(i)] There are at most a countable number of connected components of $\mathcal I$ in each $\mathcal E_0$.
\item[(ii)]  Every connected component of $\mathcal I$ is $C^\infty$-pathwise connected.
\end{enumerate}
\end{proposition}

\noindent and

\begin{corollary}
\label{countableTeich}
The Teichm\"uller and Riemann moduli stacks have at most a countable number of connected components. Moreover, 
\begin{enumerate}
\item[(i)] The natural projection map from $\mathcal I$ onto $\mathcal T(X)$ induces a bijection
\begin{equation}
\label{ccbij}
\begin{CD}
\pi_0(\mathcal I)@>1:1>>\pi_0(\mathcal T(X))
\end{CD}
\end{equation}
\item[(ii)] The mapping class group $\mathcal M\mathcal C(X)$ acts on both $\pi_0(\mathcal I)$ and $\pi_0(\mathcal T(X))$.
\item[(iii)] Passing to the quotient by the mapping class group $\mathcal M\mathcal C(X)$, the bijection \eqref{ccbij} descends as a bijection
\begin{equation}
\label{ccbij2}
\begin{CD}
\pi_0(\mathcal I)/\mathcal M\mathcal C(X)@>1:1>>\pi_0(\mathcal T(X))/\mathcal M\mathcal C(X)@>1:1>>\pi_0(\mathcal M(X)).
\end{CD}
\end{equation}
\end{enumerate}
\end{corollary}

\begin{proof}
Just use Proposition \ref{countable1} and the fact that $\text{Diff}^0(X)$ leaves the components of $\mathcal I$ invariant.
\end{proof}

For further use, we let
\begin{equation}
\label{mccc}
[\phi]\in\mathcal M\mathcal C(X)\longmapsto [\mathcal I_0\cdot \phi]\in\pi_0(\mathcal I)
\end{equation}
denote the map given by the action of the mapping class group onto a fixed component $\mathcal I_0$.

\begin{remark}
\label{countableexamples}
For surfaces, the number of connected components of $\mathcal M(X)$, that is the number of connected components of $\mathcal I$ up to the action of the mapping class group,  is finite as soon as it contains a projective manifold \cite{FM}. However, it may be more than one, see \cite{Cat2}. 
In dimension $3$, there are examples of manifolds with $\mathcal M(X)$, henceforth $\I$ having infinitely many connected components, as $\mathbb S^1\times\mathbb S^{4n-1}$ for $n>1$, see \cite{Morita}, or the product of a K3 surface with $\mathbb S^2$, see \cite{LeBrun}.
\vspace{5pt}\\
In the above examples, we note that $\mathcal E$ also has infinitely many connected components. Indeed each connected component of $\mathcal E$ contains exactly one connected component of $\mathcal I$. This leads to the following problem:
\begin{problem}
 \label{ccpb}
Find a $C^\infty$ compact manifold $X$ with $\mathcal E$ connected and $\mathcal I$ having an infinite number of connected components. 
\end{problem}

Probably, $\mathbb S^1\times\mathbb S^{4n-3}$ for $n>1$ give such an example. In particular, it is proven in \cite{Morita} that $\mathcal E$ has a single connected component. And the structures of \cite{BVdV} should give the countably many connected components of $\mathcal I$. Since they have pairwise not biholomorphic universal covers, this should give the countably many connected components of $\I$ and even of $\mathcal M(X)$. But proving this is the case seems to be out of reach for the moment. Observe that the first step in showing this result would be to establish that any deformation {\it in the large} of a Hopf manifold is a Hopf manifold, which is still an open problem as far as we know.
\vspace{5pt}\\
The case of surfaces is somewhat different, see Remark \ref{CC2}.
\end{remark}

Recall that Kodaira and Spencer defined in \cite{K-S1} the notion of {\it $c$-homotopy}. Taking into account Kuranishi's Theorem, it turns out that we may equivalently define it by saying that $J_1\in\mathcal I$ and $J_2\in\mathcal I$ are $c$-homotopic if there exists a smooth path in $\mathcal I$ joining them. That is if they belong to the same connected component $\mathcal I_0$. Similarly, we define

\begin{definition}
Let $J_1$ and $J_2$ be two points of the same $\mathcal I_0$. Then we say that they are {\it $f$-homotopic} if there exists a smooth path in $\mathcal I_0$ joining them such that the function $h^0$ is constant along it.
\end{definition}

Recall also that, if $K$ denotes the Kuranishi space of some $J_0$, then for any $c\in\mathbb N$, the sets
\begin{equation}
\label{Kc}
K^c=\{J\in K\mid h^0(J)\geq c\}
\end{equation}
are analytic subspaces of $K$, cf. \cite{Gr}. Using Kuranishi's Theorem, we immediately obtain that the sets
\begin{equation}
\label{Ic}
\mathcal I^c=\{J\in \mathcal I\mid h^0(J)\geq c\}
\end{equation}
are analytic subspaces\footnote{To be more precise, one should pass to operators of class $L^2_l$ as in Remark \ref{Sobolev} to have that $\mathcal I$ and $\mathcal I^c$ are Banach analytic spaces in the sense of \cite{Douady}.} of $\mathcal I$. Observe that $\mathcal I^c$ is the union of all $f$-homotopy classes whose $h^0$ is greater than or equal to $c$.
\vspace{5pt}\\
The analyticity of \eqref{Kc} comes indeed from the fact that the function $h^0$ is upper semi-continuous for the Zariski topology, see \cite{Gr}. But this also implies
\begin{proposition}
\label{countable2}
There are at most a countable number of $f$-homotopy classes in each $\mathcal I_0$.
\end{proposition}

Define a weighted and directed graph as follows. Each $f$-homotopy class $\mathcal F$ of $\mathcal I$ corresponds to a vertex with weight equal to $h^0(J)$ for $J\in\mathcal F$. Two vertices $\mathcal F_1$ and $\mathcal F_2$ are related by an oriented edge if there exists a smooth path $c$ in $\mathcal I$ such that
\begin{enumerate}
\item[(i)] The structure $c(0)$ belongs to $\mathcal F_1$.
\item[(ii)] For $t>0$, the structure $c(t)$ belongs to the class $\mathcal F_2$. 
\end{enumerate}
Observe that the edge is directed from the highest weight to the lowest weight.

\begin{definition}
\label{graph}
The previous graph is called {\it the graph of $f$-homotopy of $\mathcal I$}.
\end{definition}

\begin{proposition}
\label{graphproperties}
The graph of $f$-homotopy has the following properties:
\begin{enumerate}
\item[(i)] It has at most a countable number of connected components. Moreover, there is a $1:1$ correspondence between these connected components and the connected components of $\mathcal I$.
\item[(ii)] It has at most a countable number of vertices.
\item[(iii)] Each vertex is attached to at most a countable number of edges.
\item[(iv)] There is no directed loop.
\item[(v)] Every directed path is finite.
\end{enumerate}
\end{proposition}

\begin{proof}
Items (i), (ii) and (iii) come
from Proposition \ref{countable1}, Proposition \ref{countable2} and the definitions; items (iv) and (v) come from the fact that the weights are strictly decreasing along an edge.
\end{proof}

The group $\mathcal M\mathcal C(X)$ acts on the graph of $f$-homotopy. We detail in the following Proposition some trivial properties of this action.

\begin{proposition}
\label{graphaction}
The action of $\mathcal M\mathcal C(X)$ onto the graph of $f$-homotopy 
\begin{enumerate}
\item[(i)] sends a connected component onto a connected component.
\item[(ii)] sends a vertex to a vertex of same weight.
\item[(iii)] respects the number and the orientation of the edges attached to a vertex.
\end{enumerate}
\end{proposition}

Hence, the existence of diffeomorphisms acting non trivially on the graph implies strong symmetry properties of the graph. Indeed, if some $f$ sends a connected component of $\mathcal I$ onto a different one, then these two connected components of $\I$ must be completely isomorphic. 

\begin{example}{\bf Hopf surfaces.}
\label{Hopf}
Let $X=\mathbb S^3\times\mathbb S^1$. By classical results of Kodaira \cite{Ko}, \cite{BHPV}, every complex surface diffeomorphic to $X$ is a (primary) Hopf surface. There is only one connected component of complex structures up to action of the mapping class group, since any Hopf surface is $c$-homotopic to any other one, see \cite{We}. The mapping class group of $X$ is a non trivial group\footnote{\label{Hatcher} It was pointed out to me by A. Hatcher that no mapping class group of a closed $4$-manifold seems to be known.}. Indeed, observe that it contains at least the elements
$$
f(z,w)=(\bar z, \bar w)\quad\text{ and }\quad g(z,w)=(z, P(z)\cdot w)
$$
for $(z,w)\in\mathbb S^1\times\mathbb S^3\subset \mathbb C\times\mathbb C^2$ and $P$ a homotopically non trivial loop in $\text{SO}_4$, since both have non trivial action in homology. Even without knowing the mapping class group, we can characterize its action on $\mathcal I$.
Following \cite[p.24]{We}, we separate Hopf surfaces into five classes namely classes IV, III, IIa, IIb and IIc.
\begin{lemma}
\label{Hopfcc}
Let $f$ be a diffeomorphism of $X$. Assume that $f$ leaves a connected component of $\mathcal I$ invariant. Then $f$ is $C^\infty$-isotopic to the identity.
\end{lemma}

\begin{proof}
Let $J_0$ represent a Hopf surface of type IIc, that is associated to a contracting diagonal matrix
\begin{equation}
\label{matrice}
\begin{pmatrix}
\lambda_1 &0\cr
0& \lambda_2
\end{pmatrix}
\end{equation}
with $0<\vert\lambda_1\vert<\vert\lambda_2\vert<1$.
\vspace{5pt}\\
Assume that $J_0\cdot f$ belongs to the same connected component as $J_0$. Then there exists a $c$-homotopy of Hopf surfaces $\mathcal X\to [0,1]$ with endpoints $X_0$ and $X_{J_0\cdot f}$: just take the tautological family above a smooth path in $\I$ joining $J_0$ to $J_0\cdot f$.  By \cite[Theorem 8.1]{circle}, there exists an analytic space $K$ encoding the complex structures in a neighborhood of the path and obtained by gluing together a finite number of Kuranishi spaces of Hopf surfaces (up to taking the product with some vector space) such that the family $\pi$ maps onto a smooth path into $K$. Using the description of the Kuranishi spaces of Hopf surfaces in \cite[Theorem 2]{We}, it is easy to check that 
\begin{enumerate}
\item[(i)] $K$ is a manifold.
\item[(ii)] The points of $K$ encoding the type IIa Hopf surfaces belongs to a submanifold of complex codimension at least $1$.
\end{enumerate}
Hence, by transversality, we may replace the initial path defining the $c$-homotopy with a new path and a thus a new $c$-homotopy with same endpoints and such that all surfaces along this path are linear, that is not of type IIa. 
Such a family is locally and thus globally since the base is an interval isomorphic to the quotient of $\mathbb C^2\setminus\{(0,0)\}\times [0,1]$ by the action generated by
$$
(Z,t)\longmapsto (A(t)\cdot Z, t)
$$
for $A$ a smooth map from $[0,1]$ into $\text{GL}_2(\mathbb C)$ which is equal to \eqref{matrice} at $0$.
In particular, this means that $A(1)$ is conjugated to \eqref{matrice} by, say, $M$. Hence the map
\begin{equation}
\label{lineaire}
Z\in \mathbb C^2\setminus\{(0,0)\}\longmapsto M\cdot Z\in \mathbb C^2\setminus\{(0,0)\}
\end{equation}
induces a biholomorphism between $X_0$ and $X_{J_0\cdot f}$, which is smoothly isotopic to the identity. Composing $f$ with the inverse of this biholomorphism, this gives an automorphism of $X_0$ which corresponds to the same element of the mapping class group as $f$. 
\vspace{5pt}\\
Since every automorphism of every Hopf surface is isotopic to the identity (cf. \cite[p.24]{We} where all the automorphism groups are described), we are done.
\end{proof}
From Lemma \ref{Hopfcc}, we deduce that $\mathcal I$ decomposes into several identical connected components that are exchanged by action of the mapping class group. In particular,

\begin{corollary}
\label{Hopfcc2}
The map {\rm \eqref{mccc}} is a $1:1$ correspondence between the mapping class group of $X$ and the set of connected components of $\mathcal I$.
\end{corollary}

\begin{proof}
Since all Hopf surfaces are $c$-homotopic, \eqref{mccc} is surjective. And it is injective by Lemma \ref{Hopfcc}.
\end{proof}

Let us focus on one of the connected components. It corresponds to a graph with an infinite number of vertices: one of weight $4$ (class IV), one of weight $3$ for each value of $p>1$ (class III of weight $p$) and one of weight $2$ (classes IIa, IIb and IIc together). There is an edge joining $4$ to $2$ and one joining $3$ to $2$ for each value of $p$. There is no edge from $4$ to any vertex $3$ because it is not possible to deform a Hopf surface of class IV onto one of class III without crossing the $f$-homotopy class of weight $2$. In the same way, there is no edge between two different vertices of weight $3$, because every $c$-homotopy from a Hopf surface of type III with weight $p$ to a Hopf surface of type III with weight $q\not =p$ must pass through type II Hopf surfaces.

In Figure \ref{Hopffig}, we draw the graph in a synthetic way. The vertex $3p$ encodes indeed the uncountable set of vertices of weight $3$ labelled by $p>1$. The single edge from $3$ to $2$ remembers all the edges from vertices $3$ of label $p$ onto the vertex $2$.
\begin{figure}[H]
    \centering
    \includegraphics[width=0.3\textwidth]{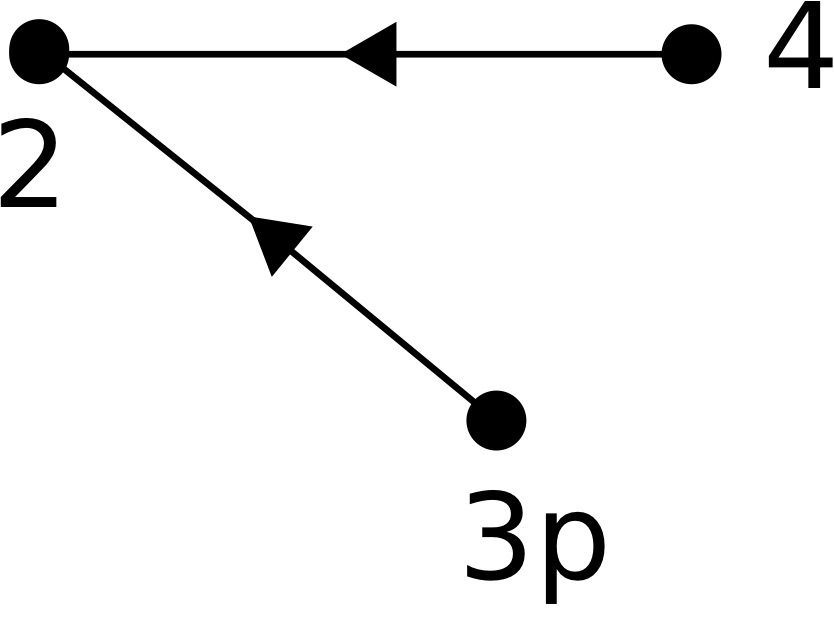}
    \caption{A component of the graph of $f$-homotopy for Hopf surfaces.}
    \label{Hopffig}
\end{figure}
\end{example}

\begin{remark}
Using the five classes of Hopf surfaces, one obtains a graph of small deformations which is more precise and complicated than the graph of $f$-homotopy, see \cite{We}, p.31. The graph of $f$-homotopy must be considered as a very rough decomposition of $\mathcal I$.
\end{remark}

\begin{example} {\bf Hirzebruch surfaces.}
\label{Hirzebruch}
Consider $X=\mathbb S^2\times\mathbb S^2$. It admits complex structures of even Hirzebruch surfaces $\mathbb F_{2a}$. By \cite{F-Q}, this exhausts the set of complex surfaces diffeomorphic to $X$. Then there is only one connected component of complex structures up to action of the mapping class group. The mapping class group is not known (cf. footnote \ref{Hatcher}) but contains at least four elements generated by
$$
f(x,y)=(a(x),a(y))\quad\text{ and }\quad g(x,y)=(y,x).
$$
where $a$ is the antipodal map of $\mathbb S^2$. Analogously to Lemma \ref{Hopfcc} and Corollary \ref{Hopfcc2}, we have
\begin{lemma}
\label{Hirzebruchcc}
Let $\phi$ be a diffeomorphism of $X$. Assume that $\phi$ leaves a connected component of $\mathcal I$ invariant. Then $\phi$ is $C^\infty$-isotopic either to $g$ or to the identity.
\end{lemma}

\begin{proof}
Let $J_0$ represent $\mathbb P^1\times\mathbb P^1$.
Assume that $J_0\cdot \phi$ belongs to the same connected component as $J_0$. Then there exists a $c$-homotopy of Hirzebruch surfaces $\pi : \mathcal X\to [0,1]$ with endpoints $X_0$ and $X_{J_0\cdot \phi}$: just take the tautological family above a smooth path in $\I$ joining $J_0$ to $J_0\cdot f$. By \cite[Theorem 8.1]{circle}, there exists an analytic space $K$ encoding the complex structures on a neighborhood of the path and obtained by gluing together a finite number of Kuranishi spaces of Hirzebruch surfaces (up to taking the product with some vector space) such that the family $\pi$ maps onto a smooth path into $K$. Using the description of the Kuranishi spaces of Hirzebruch surfaces in \cite[p.21]{Cat} (see also Example \ref{Hirzebruchbis}), it is easy to check that 
\begin{enumerate}
\item[(i)] $K$ is a manifold.
\item[(ii)] The points of $K$ encoding $\mathbb F_{2a}$ for $a>0$ belongs to a submanifold of complex codimension at least $1$.
\end{enumerate}
Hence, we may replace the initial path defining the $c$-homotopy with a new path and a thus a new $c$-homotopy with same endpoints and such that all surfaces along this path are biholomorphic to $\mathbb P^1\times\mathbb P^1$. By Fischer-Grauert's Theorem (see \cite{PS} for the version we use), such a deformation is locally trivial, hence trivial since the base is an interval, i.e. there exists a smooth isotopy of biholomorphisms
\begin{equation}
\label{isotopy}
\psi_t\ :\ \mathbb P^1\times\mathbb P^1\to \pi^{-1}(t)\qquad (t\in [0,1])
\end{equation}
In particular, $\psi_1\circ\psi_0^{-1}$
induces a biholomorphism between $X_0$ and $X_{J_0\cdot\phi}$, which is smoothly isotopic to the identity. Composing its inverse with $\phi$, this gives an automorphism of $X_0$, that is of $\mathbb P^1\times\mathbb P^1$, which corresponds to the same element of the mapping class group as $\phi$. Comparing with the automorphism group of $\mathbb P^1\times\mathbb P^1$ yields the result.
\end{proof}
\noindent and
\begin{corollary}
\label{Hirzebruchcc2}
The map {\rm \eqref{mccc}} is surjective with kernel $\{[Id], [g]\}$.
\end{corollary}

\begin{proof}
Since all Hirzebruch surfaces are $c$-homotopic, \eqref{mccc} is surjective. Lemma \ref{Hirzebruchcc} gives the kernel.
\end{proof}

Now, fix a connected component $\mathcal I_0$. We want to describe it more precisely. Observe that $g$ corresponds to an automorphism of $\mathbb P^1\times\mathbb P^1$, but not of the other Hirzebruch surfaces since every automorphism of $\mathbb F_{2a}$ is isotopic to the identity for $a>0$. Recall that the dimension of the group of automorphism of $\mathbb F_{2a}$ is $2a+5$ for $a>0$, \cite[p.44]{MK}. This implies
\begin{lemma}
\label{gaction}
We have:
\begin{enumerate}
\item[(i)] The subset $\mathcal I_0(\mathbb F_0)$ of $\mathcal I_0$ consisting of structures biholomorphic to $\mathbb P^1\times\mathbb P^1$ is open and connected.
\item[(ii)] The closed set $\mathcal I_0\setminus \mathcal I_0(\mathbb F_0)$ has exactly two connected components.
\item[(iii)] The diffeomorphism $g$ acts on $\mathcal I_0$ by fixing globally $\mathcal I_0(\mathbb F_0)$; and by exchanging the two components of $\mathcal I_0\setminus \mathcal I_0(\mathbb F_0)$.
\item[(iv)] Fix a connected component $\mathcal I_1$ of $\mathcal I_0\setminus \mathcal I_0(\mathbb F_0)$. Then the set of points $\mathcal I_2$ encoding $\mathbb F_2$ in $\mathcal I_1$ is open and connected and its complement is connected.
\item[(v)] By induction, for $a>1$, the set of points $\mathcal I_a$ encoding $\mathbb F_{2a}$ in $\mathcal I_{a-1}$ is open and connected and its complement is connected.
\end{enumerate}
\end{lemma}

\begin{proof}
Observe that $\mathcal I_0(\mathbb F_0)$ is equal to $\mathcal I_0(7)$, recall \eqref{I(k)}. Hence it is open. Also we have already observed in the proof of Lemma \ref{Hirzebruchcc} that two $c$-homotopic structures both encoding $\mathbb P^1\times\mathbb P^1$ are $c$-homotopic through a path all of whose points encode $\mathbb P^1\times\mathbb P^1$. This proves (i).
\vspace{5pt}\\
To prove (ii) and (iii), we need a variation of Lemma \ref{Hirzebruchcc}. Let $J_0$ represent $\mathbb F_2$. Call $\mathcal I_1$ the connected component of $J_0$ in $\mathcal I_0\setminus \mathcal I_0(\mathbb F_0)$.
Assume that $J_0\cdot\phi$ belongs to $\mathcal I_1$. Then there exists a smooth family of Hirzebruch surfaces $\pi : \mathcal X\to [0,1]$ with endpoints $X_0$ and $X_{J_0\cdot\phi}$ and all of whose point are distinct from $\mathbb P^1\times\mathbb P^1$. Using Theorem 8.1 of \cite{circle} and the description of the Kuranishi spaces of Hirzebruch surfaces in \cite{Cat}, p.21 (see also Example \ref{Hirzebruchbis}), it is easy to check that 
we may assume that all surfaces along this path are biholomorphic to $\mathbb F_2$. Arguing as in the proof of Lemma \ref{Hirzebruchcc}, we deduce that $\phi$ must be smoothly isotopic to the identity, since every automorphism of $\mathbb F_2$ has this property. Since we already know that $g$ fixes globally $\mathcal I_0$, this means that $J_0$ and $J_0\cdot g$ belongs to two distinct connected components of $\mathcal I_0\setminus \mathcal I_0(\mathbb F_0)$ in $\mathcal I_0$. 
\vspace{5pt}\\
Assume now that $J_1$ is another point of $\mathcal I_0$ encoding $\mathbb F_2$. Then there exists $\phi\in\text{Diff}^+(X)$ such that $J_1$ equals $J_0\cdot\phi$. By Corollary \ref{Hirzebruchcc2}, $\phi$ is either isotopic to the identity or to $g$. In the first case, $J_1$ belongs to also to $\mathcal I_1$. In the second case, it belongs to $\mathcal I_1\cdot g$. Hence, there are exactly two connected components exchanged by $g$, and items (ii) and (iii) are proved.
\vspace{5pt}\\
Finally, similar arguments prove (iv) and (v).
\end{proof}
\end{example}

In other words, the associated graph of $f$-homotopy has several connected components and each connected component has two branches joined on the vertex corresponding to $\mathbb P^1\times\mathbb P^1$. Finally, each branch has a countable number of vertices, namely one vertex for each value of $a\in\mathbb N$. It has weight $2a+5$, except for $\mathbb F_0$ which has weight $6$. Given any $a>b$, there exists an edge from $a$ to $b$, because it is possible to deform $\mathbb F_{2a}$ onto $\mathbb F_{2b}$, cf. \cite{Cat} or \cite{MK}. In particular, every vertex is attached to a countable number of edges. Similar picture is valid for the odd Hirzebruch surfaces.
 \begin{figure}[H]
 	\centering
 	\includegraphics[width=0.7\textwidth]{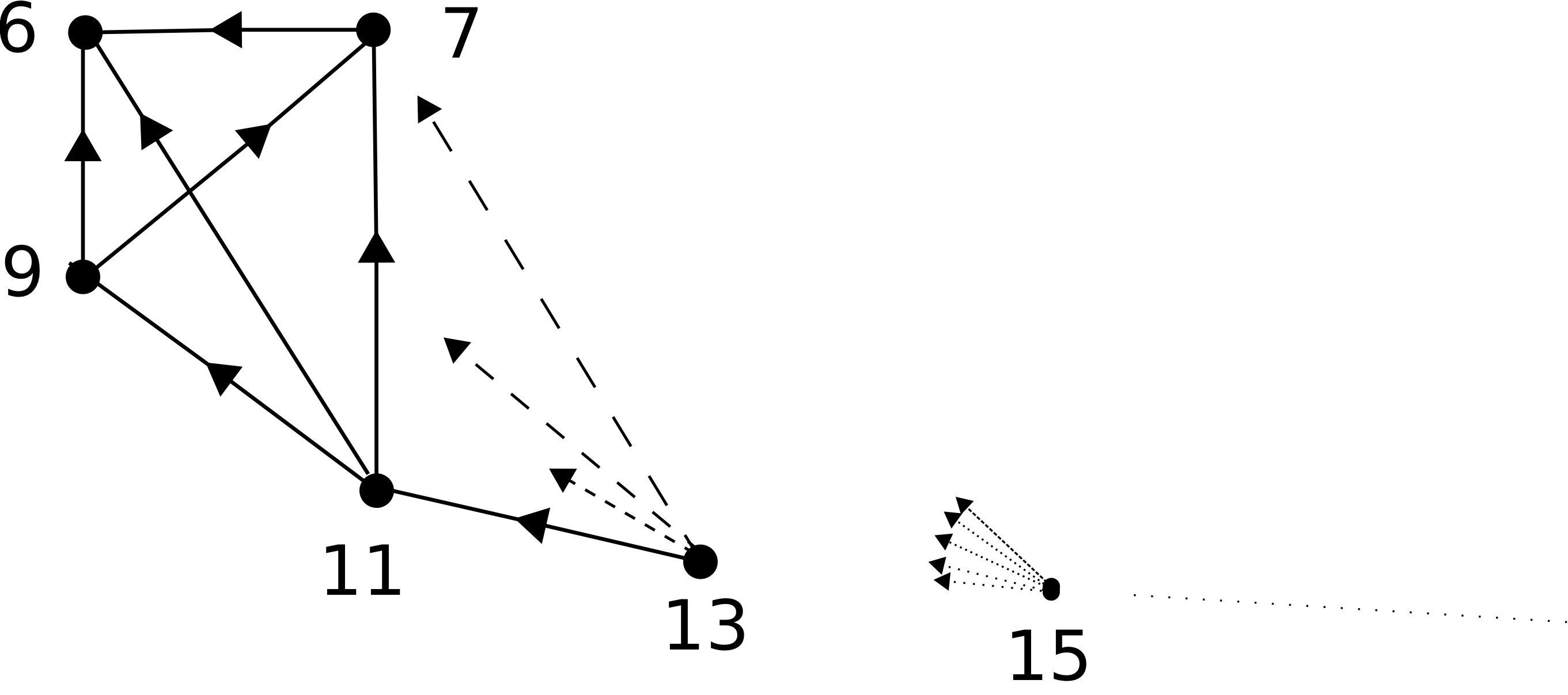}
 	\caption{One of the two branches of a component of the graph of $f$-homotopy for Hirzebruch surfaces.}
 	\label{Hirzebruchfig}
 \end{figure}

\begin{remark}
\label{ell}
Observe that the action of the mapping class group on $\I$ may take strongly different forms, depending on the $C^\infty$-manifold $X$. For $\mathbb S^3\times\mathbb S^1$, Lemma \ref{Hopfcc} shows that it only permutes the connected components of $\I$. For $\mathbb S^2\times \mathbb S^2$, some of the elements of the mapping class group permute the connected components of $\I$ but we also have by Lemma \ref{gaction} an involution which fixes each component of $\I$. Note that this involution is isotopic to an automorphism of $\mathbb P^1\times\mathbb P^1$. The case of elliptic curves shows a different phenomenon. There is a single connected component of complex structures which is fixed by every element of the mapping class group $\text{SL}_2(\mathbb Z)$. Some of them are isotopic to an automorphism of an elliptic curve, for example the multiplication by $i$; but most of them are not, cf. Example \ref{tori}. 
\end{remark}

\begin{remark}
\label{CC2}
Observe that in Examples \ref{Hopf} and \ref{Hirzebruch}, the Riemann moduli stack $\mathcal M(X)$ is connected because of Lemmas \ref{Hopfcc} and \ref{Hirzebruchcc} (cf. Remark \ref{countableexamples}). However, we do not know if $\mathcal T(X)$ has a finite number of connected components, because it is not known if the mapping class group of  $\mathbb S^1\times\mathbb S^3$, respectively $\mathbb S^2\times\mathbb S^2$, is finite or not\footnote{I owe this information to Daniel Ruberman.}. For example, notice that some blow ups of connected sums of $\mathbb P^2$s have infinite mapping class group, see \cite{Ruber}.
\end{remark}

\section{The TG foliated structure of $\mathcal I$.}
\label{TGfoliationsection}
Let $\mathcal I_0$ be a connected component of $\mathcal I$. Assume that for all $J$ in $\mathcal I_0$, we have $h^0(J)$ equal to zero. Then, the action of $\text{Diff}^0(X)$ onto $\mathcal I_0$ is locally free and one would like to conclude that it defines a foliation of $\mathcal I_0$.

This can be made precise as follows. 

\begin{proposition}
 \label{foliationsimple}
 Assume that the function $h^0$ is identically zero on the connected component $\mathcal I_0$. Then, the action of $\text{\rm Diff}^0(X)$ onto $\mathcal I_0$ induces a holomorphic foliation of $\mathcal I_0$ whose leaves are Fr\'echet submanifolds and whose local transverse section at a point $J$ is given by the Kuranishi space of $X_J$.
\end{proposition}

\begin{remark}
 \label{foldef}
 Be careful that we use the word "foliation" in an extended sense. Firstly the leaves are infinite-dimensional and secondly the transverse sections are singular spaces and are not all isomorphic. We should rather talk of  "lamination" but we prefer to reserve this terminology for foliated spaces transversely modeled onto a continuous space, e.g. a Cantor set. 
\end{remark}

\begin{proof}
The condition that the function $h^0$ is zero on the whole $\mathcal I_0$ implies that, in Theorem \ref{KuranishiTheorem}, we may take $L_0$ to be the full $A^0$. This complex vector space is, as  a real vector space, the space of vector fields $\Sigma(TX)$. Its complex structure a priori depends on the base point $J$, but it is easy to check that all $A^0$ are isomorphic as complex vector spaces, \cite[Lemma 7.1]{circle}. Hence the isomorphisms \eqref{Phi0} form a foliated atlas of $\mathcal I_0$: the plaques representing the local orbits of $\text{Diff}^0(X)$ are preserved by the changes of charts, cf. \cite[\S 6]{circle}. The leaves are Fr\'echet submanifolds modeled onto $A^0$ and at a point $J$, any germ of transverse section is isomorphic to the Kuranishi space of $X_J$.
\end{proof}

In the general case, we think of Kuranishi Theorem \ref{KuranishiTheorem} as describing a foliated structure on $\I$ which is no more transversally modelled onto an analytic space as in Proposition \ref{foliationsimple} but on the Kuranishi stack of section \ref{localaction}. In previous versions of this paper, we formalize this structure as a TG foliation, but the definition we gave is not completely satisfactory. There are several technical issues with it and solving them is unrelated to our results, so we prefer replacing it with the notion of TG foliated structure which is a purely transverse notion. We set
\begin{definition}
	\label{tgfoldef}
By {\it TG foliated structure of $\I$}, we mean a collection of Kuranishi stacks associated to a collection of Kuranishi domains which cover the whole space $\mathcal I$.
\end{definition} 
We think of it as a collection of local transversals to the $\D$-action.

\section{The holonomy groupoid of the TG foliated structure of $\mathcal I$.}
\label{holonomy}

Let $\mathcal F$ be a foliation of some analytic space. We may associate to it a holonomy groupoid as follows (\cite[\S 5.2]{Moerdijk} and \cite{Ha}). We choose a set of local transverse sections. Objects of the groupoid are points of the disjoint union of these local sections. Morphisms are generated by holonomy morphisms, obtained by following the leaves from a transverse section to another one, identifying holonomy morphisms having the same germ. It is an \'etale groupoid, which encodes the leaf space of the foliation.\vspace{5pt}\\
Having proved in Proposition \ref{foliationsimple} that the action of $\text{Diff}^0(X)$ induces a foliation of each connected component of $\mathcal I$ when $h^0$ is equal to zero, and considering in the general case the TG foliated structure of $\I$, we would like to associate to this TG foliated structure a holonomy groupoid. As in the classical case, it should be a presentation of the quotient stack, that is here of the Teichm\"uller stack.\vspace{5pt}\\
However, this is much more involved than in the classical case. The problem is that now the transverse sections are modelled onto groupoids \eqref{tgdef}, so that holonomy morphisms are stacks morphisms between Kuranishi stacks. Hence, if we just follow the same strategy, instead of building a groupoid, we end with a disjoint union of stacks and a set of local stack morphisms. It is certainly possible to turn this collection into a nice categorical structure. However, we will not follow this path since we are interested in obtaining a presentation of the Teichm\"uller stack. The crucial point is to lift holonomy morphisms between Kuranishi stacks to morphisms between Kuranishi spaces. \vspace{5pt}\\
 This lifting process will be done in four steps, in sections \ref{multifoliate}, \ref{Teichmueller} and \ref{main}.\vspace{5pt}\\
  Firstly in section \ref{multifoliate}, we construct partial foliations of $\mathcal I_0$. Partial here means that they are not defined on the whole $\mathcal I_0$ but on an open subset. We take a countable collection of such foliations whose domains of definition cover $\mathcal I_0$. Basically, the transverse structure of these foliations at some point $J$ is modeled onto the Kuranishi space of the corresponding complex manifold $X_J$. However, the jumps in the dimension of the automorphism group cause serious problems here, and we start doing the construction in the neighborhood of a $f$-homotopy class, where equidimensionality is fulfilled. Then we extend it to the whole $\mathcal I_0$, but to achieve that, we are forced to fat the smallest Kuranishi spaces to finish with all transversals of the same dimension. This fatting process was already used in \cite{circle}. \vspace{5pt}\\
  Secondly, from this set of partial foliations, we define regular atlases for this multifoliation and simple holonomy germs as the classical holonomy germs of each partial foliation. The main point is that we allow, under certain circumstancies, composition of holonomy germs coming from two different foliations. The peculiarities of a regular atlas are useful in this process. We encode all the holonomy data related to a regular atlas in a groupoid. This is however not the good groupoid to consider, especially because changing of regular atlas does not produce a Morita equivalent groupoid. All this is done in subsections \ref{regatlas}, \ref{simpleholonomy} and \ref{firstapp}. This preliminary work is essentially notational and technical, but is important to achieve the construction.\vspace{5pt}\\
  Thirdly, building on the previous sections, we construct in subsection \ref{Teichsub} the holonomy group\-oid of the TG foliated structure of $\mathcal I_0$. We call it the {\it Teichm\"uller groupoid}. Its objects are points of a disjoint union of transverse sections of partial foliations covering $\mathcal I_0$. Its morphisms are composition of the simple holonomy germs and of morphisms of type \eqref{Aut0action} on its Kuranishi space, up to an equivalence relation.\vspace{5pt}\\
  Fourthly, and last, we prove that the Teichm\"uller groupoid is an analytic smooth groupoid and a presentation of the Teichm\"uller stack in Theorem \ref{maintheorembis}, which implies Theorem \ref{maintheorem}. Basically there are two points to check. From the one hand, it must be shown that composition of simple holonomy germs and local automorphisms describes the full action of $\text{Diff}^0(X)$ onto $\mathcal I_0$. This is done in Lemma \ref{isoteich}. From the other hand, it must be shown that the source and target maps are smooth morphisms. This is essentially an adaptation of the arguments involved in the proof of Lemma \ref{analyticKurstack}. Analogously, we prove Theorem \ref{mainRbis}, which implies Theorem \ref{mainR}.\vspace{5pt}\\
  Before developing all this construction, we consider in the next section the rigidified case, in which the TG foliated structure comes from a foliation, and the Teichm\"uller groupoid an ordinary holonomy groupoid. This can be seen as a toy model for the general construction and will serve to fixing some notations and conventions.    
 
\section{The rigidified case.}
\label{rigidified}

Recall the

\begin{definition}(see \cite{Catsurvey}, Definition 12).
\label{rig}
A compact complex manifold $X_J$ is {\it rigidified} if $\text{Aut}^1(X_J)$ is equal to the identity.
\end{definition}

In that case, the map
\begin{equation}
\label{injorbit}
f\in\text{Diff}^0(X)\longmapsto J\cdot f\in\mathcal I
\end{equation}
is injective. Moreover, 

\begin{proposition}
\label{rigidconsequences}
Assume that all structures of some connected component $\mathcal I_0$ are rigidified. Then, the action of $\text{\rm Diff}^0(X)$ onto $\mathcal I$ is free and defines a foliation of $\I_0$ whose leaves are Fr\'echet manifolds modelled onto the vector space of smooth sections of $TX$ and with local transversal $K_0$ at $J_0$.
\end{proposition}

\begin{proof}
Freeness is immediate from \eqref{injorbit}. The foliation is that of Proposition \ref{foliationsimple}.
\end{proof}

In the case of Proposition \ref{rigidconsequences}, the Teichm\"uller groupoid is just the standard holonomy groupoid of the foliation. We give now a complete treatment of this case, which serves as a toy model for section \ref{Teichmueller}. We cover $\mathcal I_0$ by a collection $(U_\alpha)_{\alpha\in A}$ of open subsets. We assume that each chart $U_{\alpha}$ is a Kuranishi domain satisfying hypothesis \ref{squarehyp} associated to the following retraction map (the composition is the identity, cf. \eqref{id})
\begin{equation}
\label{id2rigid}
\begin{CD}
K_{\alpha}\hookrightarrow U_{\alpha}@> \Xi_{\alpha}>> K_{\alpha}
\end{CD}
\end{equation}
We denote by $J_\alpha$ the base point of the Kuranishi space $K_\alpha$. Observe that the index set may be assumed to be countable, due to Proposition \ref{countableTeich} and the countability of the involved topologies.\vspace{5pt}\\
Take two points $x\in K_\alpha$ and $y\in K_\beta$ belonging to the same leaf and choose a path of foliated charts joining $x$ to $y$. A holonomy germ from $x$ to $y$ is a germ of analytic isomorphism between the pointed spaces $(K_\alpha,x)$ and $(K_\beta,y)$, which is obtained by identifying along the path of foliated charts points belonging to the same leaf, see \cite[\S 2.1]{Moerdijk} or \cite{CLN}.\vspace{5pt}\\
They can be encoded in a holonomy groupoid \cite[\S 5.2]{Moerdijk} or \cite{Ha} as follows. {\it Objects} are points of the disjoint union of transversals
\begin{equation}
 \label{objectshg}
 \uniondisjointe_{\alpha\in A} K_\alpha.
\end{equation}
We denote by $(x,\alpha)$ a point of $K_\alpha$. To encode the morphisms, we first notice that on each non-empty intersection $U_{\alpha}\cap U_{\beta}$,  there exists a unique isomorphism $\phi_{\alpha,\beta}$ between some open subset $K_{\alpha, \beta}$ of $K_{\alpha}$ and some open subset $K_{\beta, \alpha}$ of $K_{\beta}$. It is obtained by following the leaves of the foliation from $K_{\alpha}$ till meeting $K_{\beta}$ (when this occurs). It satisfies the commutative diagram
\begin{equation}
\label{CDhol}
\begin{CD}
U_{\alpha}\cap U_{\beta} @>Id>>U_{\alpha}\cap U_{\beta}\cr
@V \Xi_{\alpha} V V@ V V \Xi_{\beta}V\cr
K_{\alpha,\beta}@>\simeq>\phi_{\alpha,\beta}> K_{\beta,\alpha}
\end{CD}
\end{equation}

\begin{remark}
\label{reducedrigidified}
It happens that Kuranishi spaces are everywhere non-reduced. Hence a morphism between Kuranishi spaces is not completely determined by its values, the values of its differential must also be prescribed. The previous definition of $\phi_{\alpha,\beta}$ by following the leaves just determines its values. However, since $\Xi_\alpha$ and $\Xi_\beta$ are smooth morphisms by Kuranishi's Theorem \ref{KuranishiTheorem}, the equality $d\phi_{\alpha,\beta}\circ d\Xi_\alpha=d\Xi_{\beta}$ coming from \eqref{CDhol} determines the values of its differential. Thus, even in this non-reduced situation, we have completely and uniquely defined the isomorphism $\phi_{\alpha,\beta}$ making \eqref{CDhol} commutative.
\end{remark}

We now look at the groupoid of germs generated by the $\phi_{\alpha,\beta}$. In other words, we now let $(\alpha_1,\hdots,\alpha_n)$ be a collection of indices such that each $U_{\alpha_i}\cap U_{\alpha_{i+1}}$ is non-empty and define
\begin{equation}
\label{phialpha1n}
\phi_{\alpha_1,\hdots,\alpha_n}:=\phi_{\alpha_{n-1},\alpha_n}\circ\hdots\circ \phi_{\alpha_1,\alpha_2}.
\end{equation}
This composition is defined on some open subset of $K_{\alpha_1}$ that we denote by $K_{\alpha_1,\hdots,\alpha_n}$; and it ranges in some open subset of $K_{\alpha_n}$, that we denote by $K_{\alpha_n,\hdots, \alpha_1}$. Then we represent all holonomy maps as
points  of 
\begin{equation}
\label{morphismholonomyrigidified}
\uniondisjointe_{n\geq 1}\left (\uniondisjointe_{(\alpha_1,\hdots,\alpha_n) \in B_{n}} K_{\alpha_1,\hdots,\alpha_n}\right ).
\end{equation}
Here $(\alpha_1,\hdots,\alpha_n) \in B_{n}$ if each $U_{\alpha_i}\cap U_{\alpha_{i+1}}$ is non-empty. 
A point $x$ in some $K_{\alpha_1,\hdots,\alpha_n}$ represents the germ at $x$ of the map $\phi_{\alpha_1,\hdots,\alpha_n}$, the case $n=1$ encoding the identity 
germs. We denote such a point by the $(n+1)$-uple $(x,\alpha_1,\hdots,\alpha_n)$.\vspace{5pt}\\
Consider the groupoid whose objects  are given in \eqref{objectshg}, and morphisms are given in \eqref{morphismholonomyrigidified}. Observe that both sets are $\mathbb C$-analytic spaces. The source map sends $(x,\alpha_1,\hdots \alpha_n)$ onto $(x,\alpha_1)$ and the target map sends it to $(\phi_{\alpha_1,\hdots,\alpha_n}(x),\alpha_n)$. Both are obviously \'etale analytic maps, since the source map is just the inclusion $K_{\alpha_1,\hdots,\alpha_n}\subset K_{\alpha_1}$ on the component $K_{\alpha_1,\hdots,\alpha_n}$\footnote{This component has no reason to be connected.}; and the target map on the same component is the composition of the isomorphism $\phi_{\alpha_1,\hdots,\alpha_n}$ from $K_{\alpha_1,\hdots,\alpha_n}$ onto $K_{\alpha_n,\hdots,\alpha_1}$ with the inclusion $K_{\alpha_n,\hdots,\alpha_1}\subset K_{\alpha_n}$. Multiplication is given by composition of holonomy germs.\vspace{5pt}\\
However, we are not finished yet. The previous groupoid is not the holonomy groupoid of the foliation. We must still identify identical germs. It may happen for example that such a composition $\phi_{\alpha_1,\hdots,\alpha_n}$ is the identity. So we take the quotient of \eqref{morphismholonomyrigidified} by the following equivalence relation
\begin{equation}
\label{equivrigid}
(x,\alpha)\sim (x',\alpha')\iff
\left\{
\begin{aligned}
&x=x',\ \alpha_1=\alpha'_1,\ \alpha_n=\alpha'_{n'}\cr
\text{and }&\big (\phi_{\alpha_1,\hdots,\alpha_n}\big )_{x'}\equiv \big (\phi_{\alpha'_1,\hdots,\alpha'_{n'}}\big )_x
\end{aligned}
\right .
\end{equation}
that is if they have same source, same target, and are equal as germs.
Hence, the set of morphisms is 
\begin{equation}
\label{morphismholonomyrigid2}
\uniondisjointe_{n\geq 1}\left (\uniondisjointe_{(\alpha_1,\hdots,\alpha_n) \in B_{n}} K_{\alpha_1,\hdots,\alpha_n}\right )\Biggm /\sim
\end{equation}
We set
\begin{definition}
	\label{teichrigid}
	We call {\it Teichm\"uller groupoid} of $\mathcal I_0$ the groupoid whose objects are given by \eqref{objectshg}, whose morphisms are given in \eqref{morphismholonomyrigid2}, and whose source, target maps and multiplication are defined as above.
\end{definition}
We define in the same way the Teichm\"uller groupoid of $V$, an open subset of $\I$.
\begin{proposition}
	\label{groupoidrigid}
	The Teichm\"uller groupoid is an analytic \'etale groupoid. 
\end{proposition}

\begin{proof}
From the above discussion, we just have to prove that \eqref{morphismholonomyrigid2} is still an analytic space and that the projection map from \eqref{morphismholonomyrigidified} onto \eqref{morphismholonomyrigid2} is \'etale.

Observe that two distinct points of the same component $K_{\alpha_1,\hdots,\alpha_n}$ of \eqref{morphismholonomyrigidified} cannot be equivalent. Therefore, the natural projection
map from \eqref{morphismholonomyrigidified} onto \eqref{morphismholonomyrigid2} is \'etale and we just have to show that \eqref{morphismholonomyrigid2} is Hausdorff to finish with the proof.

This comes from a standard argument, cf. \cite[prop. 3.2]{Brunella}. Consider two equivalent convergent sequences $(x_p,\alpha_p)$ and $(x_p, \alpha'_p)$. We may assume that all $\alpha_p$, resp. $\alpha'_p$, are the same, say $\alpha$, resp. $\alpha'$. Assume that $(x_p)$ converges to $x$. Then this means that $\big (\phi^{-1}_{\alpha_1,\hdots,\alpha_n}\circ \phi_{\alpha'_1,\hdots,\alpha'_n}\big )_{x_p}\equiv Id_{x_p}$, i.e. the germ of this morphism is the identity at every point $x_p$. By analyticity, this implies that it is also the identity at the limit point $x$.
\end{proof}

\begin{remark}
\label{moritaequivrigid}
The construction above depends on a choice of a foliated atlas. However, it is easy to show that it is independent of this choice up to Morita equivalence. This can of course be deduced from general arguments, since it represents the stack $\mathcal T(X,\I_0)$, which does not depend on a foliated atlas. It can also be proved directly as follows. Start with a foliated atlas and construct the associated Teichm\"uller groupoid. Take a finer foliated atlas. Then the associated Teichm\"uller groupoid is just the localization of the first one over the new atlas, hence both are weakly equivalent \cite{Ha}. Start now with two different foliated atlases and their associated Teichm\"uller groupoid. Since the union of the atlases is a common refinement of both of them, the two groupoids are Morita equivalent.
\end{remark}
 
\begin{remark}
\label{Aut10bis}
Assume that for all structures $J$ in $\mathcal I_0$, we have $\text{Aut}^0(X_J)$ equal to the identity. Then Proposition \ref{foliationsimple} still applies and the action of $\text{Diff}^0(X)$ still defines a foliation of $\mathcal I_0$. So we can still define a holonomy groupoid as above. Morover the geometric quotient of the Teichm\"uller stack equals the leaf space, that is the geometric quotient of this holonomy groupoid. Nevertheless, they may be different as {\it stacks}, because there may exist a non trivial element in $\text{Aut}^1(X_J)$ that fixes $\mathcal I_0$. Such an element is encoded in the Teichm\"uller groupoid we construct in section \ref{Teichmueller} but not in the holonomy groupoid of Definition \ref{teichrigid}, cf. Remark \ref{subtle}.\vspace{3pt}\\
For many compact complex manifolds $X_0$, there is no difference between $\text{Aut}^0(X_0)$ and $\text{Aut}^1(X_0)$, cf. \cite{Catsurvey}. We gave an example of $X_0$ with $\Azero$ and $\Aun$ distinct in \cite{Aut1paper}. The dimension of $\Azero$ is positive so this leads to the following problem.

\begin{problem}
Find a compact complex manifold $X_0$ with $\text{Aut}^0(X_0)$ being reduced to the identity but which is not rigidified.
\end{problem}

If $X_0$ is K\"ahler, then a result of Liebermann implies that $\text{Aut}^0(X_0)$ has finite index in $\text{Aut}^1(X_0)$\footnote{I owe this information to S. Cantat.}. In the non-K\"ahler case, however, there should even  exist examples with infinite "complex mapping class group"  $\text{Aut}^1(X_0)/\text{Aut}^0(X_0)$.
\end{remark}

\section{The set of partial foliated structures of $\mathcal I$.}
\label{multifoliate}

In this section, we associate to the TG foliated structure of a connected component $\mathcal I_0$ of $\mathcal I$ a collection of standard foliations of open sets of $\mathcal I_0$ covering it. In subsection \ref{simpleholonomy}, we will associate to these partial foliations their holonomy germs. This is a crucial step in defining the morphisms of the Teichm\"uller groupoid. The main problem here is that the dimension of the Kuranishi spaces may vary inside $\mathcal I_0$. To overcome this difficulty, we proceed in two steps. It turns out that the dimension we have really to care about in this problem is the dimension of the automorphism group. Hence we first work in the neighborhood of a $f$-homotopy class, so that we may assume equidimension of the automorphism groups involved in the choice of foliated atlases. Then, we treat the general case. We have to fat the Kuranishi spaces with small automorphism group, following a process already used in \cite{circle}. This supposes the function $h^0$ to be bounded on $\mathcal I_0$. 

\subsection{The set of partial foliated structures of a neighborhood of a $f$-homotopy class.}
\label{multifoliatefclass}
Let $\mathcal F$ be a $f$-homotopy class in $\mathcal I$. Let $V$ be a connected neighborhood of $\mathcal F$ in $\mathcal I_0$. Let $G(\Sigma(TX))$ be the grassmannian of closed vector subspaces of $\Sigma (TX)$ of codimension $h^0(\mathcal F)$. For each $L\in G(\Sigma (TX))$,
define
\begin{equation}
\label{FL}
\mathcal F_L=\{J\in\mathcal F\mid L\oplus \text{Re }H^0(X_J,\Theta_J)=\Sigma(TX)\}.
\end{equation}

\begin{definition}
\label{Fadmissible}
We say that $L$ is {\it $\mathcal F$-admissible} if $\mathcal F_L$ is not empty.
\end{definition}

Assume that $L$ is $\mathcal F$-admissible and let $J_0\in\mathcal F_L$. Then, using the isomorphism 
\begin{equation}
\label{real}
\xi\in A^0\longmapsto \xi+\bar\xi\in \Sigma (TX)
\end{equation}
(where $A^0$ is the space of $(1,0)$-vectors for the structure $J_0$), we see that the choice of a $\mathcal F$-admissible $L$ is equivalent to the choice of a closed subspace $L_0$ of $A^0$ satisfying \eqref{A0} and 
\begin{equation}
\label{L0L}
\text{Re }L_0=L.
\end{equation}
In the sequel, we will denote by the same symbol $L$ a closed subspace of $A^0$ and its real part in $\Sigma (TX)$. No confusion should arise from this abuse of notation. Observe that all such $L$ are complex isomorphic, cf. \cite[Lemma 7.1]{circle}.
\vspace{5pt}\\
So, once chosen such an $L$, we may apply Theorem \ref{KuranishiTheorem} at $J_0$ with $L$. 
We define $V_L$ as the maximal open subset of $V$ covered by Kuranishi domains modelled on $L$ and based at points of $\mathcal F_L$. We can interpretate it as follows. Theorem \ref{KuranishiTheorem} endows each Kuranishi domain with a trivial local foliation by copies of $L$ and leaf space $K_0$. 
\vspace{5pt}\\
Now, let us put this interpretation in a global setting. It tells us that we may cover $V_L$ by Kuranishi domains modelled on the same $L$.
Hence $L$ defines a foliation of $V_L$ by leaves locally isomorphic to a neighborhood of $0$ in $L$, see \cite[Theorem 7.2]{circle}\footnote{The assumption of compacity in this Theorem is only used to {\it prove} that there exists a common $L$ modelling all the Kuranishi domains. Since we assume the existence of such a common $L$, the proof applies.}.

\begin{definition}
\label{Lfoliation}
We call this foliation the {\it $L$-foliation of $V$} (even if it is only defined on $V_L$).
\end{definition}

In the case where $V_L$ is equal to $V$,  which is equivalent to saying that $L$ is a common complement to all $H^0(X_J,\Theta_J)$ for $J\in \mathcal F$, then we obtain a global foliation of $V$.
\vspace{5pt}\\
Nevertheless, it is not possible in general to assume this hypothesis. Hence we shall replace this foliated structure by a collection of {\it partial foliations} encoded in a groupoid.
\begin{definition}
\label{coveringfamily}
A set $\mathcal L$ of $\mathcal F$-admissible elements of $G(\Sigma (TX))$ such that
\begin{equation}
\label{CF}
\union_{L\in\mathcal L} V_L=V.
\end{equation}
is called {\it a covering family} of $V$.
\end{definition}

Choose $\mathcal L$ a covering family of $\mathcal F$. Observe that we may assume $\mathcal L$ to be countable by Proposition \ref{countable1}. To $\mathcal L$ is associated {\it a covering set of partial foliations} of $V$, defined as the set of all $L$-foliations of $V$ for $L\in\mathcal L$. It is useful to encode it in a groupoid as follows.
\vspace{5pt}\\
For each $L\in\mathcal L$, choose an atlas 
\begin{equation}
\label{atlas1}
\mathcal U_L=(U_{\alpha})_{\alpha\in A_L}
\end{equation} 
of $V_L$ by $L$-foliated charts satisfying hypothesis \ref{squarehyp}. Define
\begin{equation}
\label{A}
A=\uniondisjointe_{L\in\mathcal L}A_L\qquad\text{ and }\qquad \mathcal U=(\mathcal U_L)_{L\in\mathcal L}
\end{equation}

Once again, we may assume that $A$ is countable, due to the countability of the involved topologies. Then define the groupoid $G_{\mathcal U}$ as follows. 
{\it Objects} are points of the disjoint union 
\begin{equation}
\label{objects}
\uniondisjointe _{\alpha\in A} U_{\alpha}
\end{equation}
hence are encoded by couples $(x,\alpha)$.
\vspace{5pt}\\
We insist on seeing each $U_{\alpha}$ as a $L$-foliated Fr\'echet space. We use the notation
\begin{equation}
\label{Lalpha}
L\in\alpha
\end{equation}
to denote the vector space $L$ associated to $\alpha$. In section \ref{Teichmueller}, we will enlarge our index set $A$ and the interest of this strange notation should be clarified. Set now
\begin{equation}
\label{B}
B=\uniondisjointe_{L\in\mathcal L} B_L=\uniondisjointe_{L\in\mathcal L}\{(\alpha,\beta)\in A^2\mid \alpha\not = \beta,\ L\in\alpha\text{ and }L\in\beta\}.
\end{equation}
{\it Morphisms} are
points
\begin{equation}
\label{morphisms}
\uniondisjointe_{\alpha\in A}U_\alpha\uniondisjointe _{(\alpha,\beta)\in B} U_{\alpha}\cap U_{\beta}
\end{equation}
encoded by triples $(x,\alpha,\beta)$.
\vspace{5pt}\\
Once again, we insist on seeing each $U_{\alpha}\cap U_{\beta}$ as a $L$-foliated Fr\'echet space. Note that there is no morphism between a point in a $L$-foliated chart and the same point in a $L'$-foliated chart. 

\subsection{The general case.}
\label{multifoliategeneral}
We now deal with the definition of a covering set of partial foliations and its encoding in a groupoid for all points of $\mathcal I_0$ with bounded function $h^0$. \vspace{5pt}\\
Let $a\in\mathbb N$. Recall \eqref{I(k)}. Recall that $\mathcal I(a)$ is open. We assume that it is connected, replacing it with a connected component otherwise. Given a closed subspace $L$ of $\Sigma(TX)$ of codimension $a$, define
\begin{equation}
\label{FLbis}
\mathcal F_L=\{J\in\mathcal I_0(a)\mid L\cap\text{Re }H^0(X_J,\Theta_J)=\{0\}\}.
\end{equation}
This is an extension of \eqref{FL}. We may go on with this generalization.

\begin{definition}
We say that $L$ is {\it $a$-admissible} if $\mathcal F_L$ is not empty.
\end{definition}

Analogously to what happens in subsection \ref{multifoliatefclass}, the choice of an $a$-ad\-miss\-ible $L$ is equivalent to the choice of a closed subspace $L_0$ of $A^0$ satisfying 
\begin{equation}
\label{L0bis}
L_0\cap H^0(X_J,\Theta_J)=\{0\}\quad\text{ and }\quad\text{Re }L_0=L.
\end{equation}
As in subsection \ref{multifoliatefclass}, we denote both $L$ and $L_0$ by the same symbol $L$. Although this $L$ is not a complement of $H^0(X_J,\Theta_J)$, we may run the proof of Kuranishi's Theorem after adding some finite-dimensional subspace $H_L$ such that
\begin{equation}
\label{L+H}
L\oplus H_L\oplus H^0(X_J,\Theta_J)=A^0.
\end{equation}
\begin{remark}
 \label{Cinfini}
We assume that $H_{L}$ contains only $C^\infty$ elements, so that we may use the same $H_L$ for all Sobolev classes. This is always possible by perturbing a little a basis of $H_L$ since $C^\infty$ diffeomorphisms are dense in $L^2_l$ diffeomorphisms for $l$ big enough.
\end{remark}
We thus obtain an isomorphism between a neighborhood $U$ of $J$ in $\mathcal I$ and a product (cf. \cite[Theorem 7.2]{circle})
\begin{equation}
\label{fatiso}
\begin{CD}
U@ >\Phi:=(\Xi,\Upsilon) >> (K_J\times H_L)\times L.
\end{CD}
\end{equation}
whose inverse is given by
\begin{equation}
\label{fatisoinverse}
(J,\xi,\xi')\in \Phi(U)\cap (K_J\times H_L\times L)\longmapsto (J\cdot e(\xi))\cdot e(\xi')
\end{equation}
Setting
\begin{equation}
\label{K}
K:=\Phi(U)\subset K_J\times H_L
\end{equation}
we obtain a sequence analogous to \eqref{id2rigid}
\begin{equation}
\label{id2bis}
\begin{CD}
K\hookrightarrow U@>\Xi >> K.
\end{CD}
\end{equation}
This is our new definition of Kuranishi domains and charts. We replace hypothesis \eqref{squarehyp} with
\begin{hypothesis}
	\label{squarehyp2}
	The image of $\Phi$ is contained in a product $K\times (W'\cap L)=K_J\times (W'\cap H_L)\times (W'\cap L)$ with $W'\subset W$ an open and connected neighborhood of $0$ in $A^0$.
\end{hypothesis}
Let $\mathcal U$ be a covering of $\mathcal I_0(a)$ by Kuranishi domains satisfying hypothesis \ref{squarehyp2}.
Set $V=\mathcal I_0(a)$.
We define $V_L$ as the maximal open subset of $V$ covered by Kuranishi domains satisfying hypothesis \ref{squarehyp2}, modelled on $L$ and based at points of $\mathcal I_0(a)$. We may then define the sets of objects and morphisms of the groupoid $G_{\mathcal U}$ of partial foliations of $V$ exactly as in subsection \ref{multifoliatefclass}. The structure maps are the obvious ones (cf. the proof of Proposition \ref{groupoidI}).

\begin{remark}
\label{newK}
 Recall that the local transversal section at some point $J_0$ is not always its Kuranishi space $K_0$. It is if and only if $h^0(J_0)$ is equal to $a$. More generally, it is the product of $K_0$ with an open neighborhood of $0$ in $\mathbb C^{a-h^0(J_0)}$.
\end{remark}

\begin{remark}
 \label{CaETI}
 Observe that, if the function $h^0$ is bounded on a connected component $\mathcal I_0$ by some integer $a$, then $\mathcal I_0(a)$ is equal to $\mathcal I_0$.
\end{remark}

\subsection{Properties of the groupoid of partial foliated structures.}
\label{groupoidpfs}

The following Proposition shows that the groupoid of partial foliated structures really describes an intrinsic geometric structure. 
\begin{proposition}
\label{groupoidI}
We have:
\vspace{3pt}\\
I. The groupoid $G_{\mathcal U}$ is a foliated Fr\'echet \'etale groupoid, that is 
\begin{enumerate}
\item[(i)] Both the set of objects and that of morphisms are foliated Fr\'echet manifolds.
\item[(ii)] The source, target, composition, inverse and anchor maps are analytic and respects the foliations.
\item[(iii)] The source and target maps are local foliated isomorphisms.
\end{enumerate}
II. The foliated Fr\'echet groupoid $G_{\mathcal U}$ is independent of $\mathcal U$ up to foliated analytic Morita equivalence. 
\end{proposition}

\begin{proof}
This is completely standard, since this groupoid is very close to the Lie groupoid obtained by localization of a smooth manifold over an atlas, see \cite{Ernesto}, \S 7.1.3. Starting with I, then (i) is obvious from \eqref{objects} and \eqref{morphisms}; the source map $\sigma$ and the target map $\tau$ are given by the following foliation preserving inclusions 
\begin{equation}
\label{folinclusions}
\begin{CD}
U_{\alpha}@<\sigma << U_{\alpha}\cap U_{\beta} @>\tau >> U_{\beta}
\end{CD}
\end{equation}
proving (iii) and part of (ii). Composition is given by 
\begin{equation}
\label{composition}
(x,\alpha,\beta)\times (x,\beta,\gamma)\longmapsto (x,\alpha,\gamma)
\end{equation}
provided that
$$
L\in\alpha\cap\beta\cap\gamma
$$
(the notation should be clear from \eqref{Lalpha}). Assume for simplicity that $\alpha$, $\beta$ and $\gamma$ are pairwise distinct.
This is indeed a foliation preserving analytic map from
$$
\{(\phi,\psi)\text{ morphisms of }G_{\mathcal U}\mid \tau (\phi)=\sigma (\psi)\}
$$
that is 
\begin{equation}
\label{compdomain}
\uniondisjointe_{(\alpha,\beta,\beta,\gamma)\in \sqcup {B_L}^2} U_{\alpha}\cap U_{\beta}\cap U_{\gamma}
\end{equation}
onto \eqref{morphisms}. Other cases are treated similarly. This finishes the proof of (ii), hence of I.
\vspace{5pt}\\
As for II, start from choosing two coverings $\mathcal U$ and $\mathcal V$ of $V$.
The crucial point is contained in I: these groupoids are \'etale. From that, it is enough to observe that both the localization of $G_{\mathcal U}$ over $\mathcal V$ and the localization of $G_{\mathcal V}$ over $\mathcal U$ are equal to the groupoid $G_{\mathcal U\cap \mathcal V}$ (see \cite{Ha} for the equivalence with the classical definition of Morita equivalence).
\end{proof}
To finish this section, we note that $G_{\mathcal U}$ encodes all the possible foliations of open sets of $V$ associated to Kuranishi domains. Indeed we have

\begin{proposition}
\label{full}
The full subgroupoid of $G_{\mathcal U}$ obtained by restriction to a fixed $L\in\mathcal L$ is the localization over some atlas $V_L$, hence is Morita equivalent to the largest subdomain of $V$ foliated by $L$.
\end{proposition}

\section{The Teichm\"uller groupoid.}
\label{Teichmueller}

In this section, we construct for the TG foliated structure of $\I_0$ the analogue for the holonomy groupoid. We call it {\it the Teichm\"uller groupoid}. This will be done in several steps. In subsection \ref{regatlas}, we first give a sort of foliated atlas of $\I_0$ with good properties. We call it {\it a regular atlas}. We then define in subsection \ref{simpleholonomy} the holonomy germs associated to the set of partial foliations. In subsection \ref{firstapp}, we encode these simple holonomy morphisms in a groupoid $K_\mathcal U$. This is however not the right analogue for the holonomy groupoid, since it does not take into account the isotropy groups of the transverse structure of the TG foliated structure.  From the regular atlas, we finally build in subsection \ref{Teichsub} the Teichm\"uller groupoid.

\subsection{Regular atlases.}
\label{regatlas}
We need to construct on $V$ an equidimensional atlas from the atlas $\mathcal U$ of $K_\mathcal U$.
Besides, we need this atlas to reflect the partial foliated structure of $\mathcal I_0$ to be able to define properly the holonomy germs. \vspace{5pt}\\
As in section \ref{multifoliate}, we fix $\mathcal L$ and we define \eqref{atlas1} and \eqref{A} as well as $G_{\mathcal U}$. 
\vspace{5pt}\\
We assume that each chart $U_{\alpha}$ is a Kuranishi domain satisfying hypothesis \ref{squarehyp2}, based at $J_\alpha$ and associated to the following retraction map (the composition is the identity, cf. \eqref{id2bis})
\begin{equation}
\label{id2}
\begin{CD}
K_{\alpha}\hookrightarrow U_{\alpha}@> \Xi_{\alpha}>> K_{\alpha}
\end{CD}
\end{equation}
Recall Remark \ref{newK}.\vspace{5pt}\\
The set of holonomy germs of $G_{\mathcal U}$ is constructed from the union of all holonomy groupoids when $L$ varies. But in order to mix these holonomies, we first add some charts with common transversal for different foliations. More precisely, for every couple $(L,L')$ in $\mathcal L^2$ with
\begin{equation}
\label{inter}
V_L\cap V_{L'}\not =\emptyset
\end{equation}
we enlarge the index set $A$ to include new indices $\alpha$ and new charts  
\begin{equation}
\label{atlas2}
\begin{CD}
K_{\alpha}@<\Xi_{\alpha, L} << U_{\alpha,L}\quad\text{ and }\quad U_{\alpha,L'} @>\Xi_{\alpha, L'}>>K_{\alpha}
\end{CD}
\end{equation}
which cover \eqref{inter}. We emphasize that {\it the same} analytic set $K_{\alpha}$ is used as leaf space for both the $L$ and the $L'$-foliations. This is possible due to the uniqueness properties in Kuranishi's Theorem \ref{KuranishiTheorem}.
\vspace{5pt}\\
In the same way, for any value of $n\geq 3$, we enlarge the index set $A$ to include new indices and charts
\begin{equation}
\label{atlas3}
\begin{CD}
U_{\alpha, L_i}@>\Xi_{\alpha, L_i}>>K_{\alpha}
\end{CD}
\end{equation}
for $i=1,\hdots, n$, covering
\begin{equation}
\label{inter2}
V_{L_1}\cap\hdots\cap V_{L_n}\not =\emptyset.
\end{equation}
Once again, we insist on the fact that $K_{\alpha}$ is a {\it common leaf space} for every $L_i$-foliation restricted to $U_{\alpha,L_i}$.
We use the notation
\begin{equation}
\label{Lialpha}
L_i\in\alpha\qquad\text{ for all }\qquad  i=1,\hdots, n
\end{equation}
as a natural extension of \eqref{Lalpha}.
\medskip\\
All new charts are supposed to satisfy \ref{squarehyp2}. We define

\begin{definition}
	\label{regatlasdef}
	We call {\it regular atlas} of $V$ such a foliated atlas $\mathcal U$.
\end{definition}

\begin{remark}
\label{atlases}
It is important to notice that the new covering $\mathcal U$ is constructed from the covering $\mathcal U$ of $G_{\mathcal U}$ but has strictly more charts because of \eqref{atlas3} and \eqref{atlas2}. Moreover, this (extended) covering cannot be used to construct some $G_{\mathcal U}$, since each chart of $G_{\mathcal U}$ has to be explicitely associated to a unique $L\in\mathcal L$. However, to avoid cumbersome notations, we use the same symbol for both coverings.
\end{remark}  

We have now to pay attention to the fact that $K_\alpha$ is no more the Kuranishi space of $J_\alpha$, but its product with some open set in $H_L\simeq\mathbb C^{a-h^0(J_\alpha)}$, cf. \eqref{fatiso}. Hence the groupoid \eqref{tgdef} of subsection \ref{localaction} is not the good one to consider. This can be easily fixed by fatting also the group $\Azero$. Recall \eqref{L+H} and Remark \ref{Cinfini}. 

The following generalization of Lemmas \ref{Gdecompolemma} and \ref{Gdecompolemmaanalytic} is straightforward to prove.

\begin{lemma}
 \label{newGdecompolemma}
We have
\begin{enumerate}
	\item [(i)] If $W'\subset W$ is small enough, then there exist an open and connected neighborhood $T_\alpha$ of the identity in $\text{\rm Aut}^0(X_\alpha)$ and an open and connected neighborhood $D_{\alpha,L}$ of the identity in $\text{\rm Diff}^0(X)$ such that  
	\begin{equation}
	\label{newGdecompo}
	(\xi,\xi',g)\in W'\cap H_L\cap W'\cap L\times T_\alpha\longmapsto g\circ e(\xi)\circ e(\xi')\in D_{\alpha,L}
	\end{equation}
	is an isomorphism. 
	\item[(ii)] Set $\mathcal D_{\alpha,L}=\bigcup_{g\in \text{\rm Aut}^0(X_\alpha)}gD_{\alpha,L}$. Then \eqref{newGdecompo} extends as an isomorphism
	\begin{equation}
	\label{newGdecompoglobal}
	(\xi,\xi',g)\in W'\cap H_L\cap W'\cap L\times \text{\rm Aut}^0(X_\alpha)\longmapsto g\circ e(\xi)\circ e(\xi')\in \mathcal D_{\alpha,L}
	\end{equation}
	\item[(iii)] Both {\rm \eqref{newGdecompo}} and {\rm\eqref{newGdecompoglobal}} extend to analytic isomorphisms of the Sobolev completions. 
\end{enumerate}

\end{lemma}

Now define
\begin{equation}
\label{newG}
G_\alpha:=\{g\circ e(\xi)\mid (g,\xi)\in\text{Aut}^0(X_\alpha)\times (H_L\cap W')\}
\end{equation}

\begin{remark}
	\label{notagroup}
	Be careful that $G_\alpha$ is not a group, just a fatting of $\text{Aut}^0(X_\alpha)$.
\end{remark}
We let $g\in G_\alpha$ act on $K_\alpha$ exactly as in \eqref{Aut0action}, that is
\begin{equation}
 \label{newGaction}
 xg:=\Xi_{\alpha, L}(x\cdot g)
\end{equation}
and form the corresponding groupoid $\mathcal A_{\alpha,L}\rightrightarrows K_\alpha$ as in section \ref{localaction}. Notice that \eqref{newGaction} depends on a choice of $L$.

\subsection{Simple holonomy morphisms.}
\label{simpleholonomy}

In this subsection, we associate to the partial foliations of $\mathcal I_0$ their holonomy germs. The main point is how to mix the holonomies of the different foliations. We refer to section \ref{rigidified} for comparison.\vspace{5pt}\\
We start with a regular atlas $\mathcal U$. On each intersection $U_{\alpha}\cap U_{\beta}$ with
\begin{equation}
\label{inter3}
\alpha\cap\beta\not=\emptyset
\end{equation}
and for every choice of $L_i$ in \eqref{inter3},  we define the holonomy isomorphism $\phi_{\alpha,\beta, L_i}$ between some open subset $K_{\alpha, \beta, L_i}$ of $K_{\alpha}$ and some open subset $K_{\beta, \alpha, L_i}$ of $K_{\beta}$ as in section \ref{rigidified}. Recall the commutative diagram \eqref{CDhol}.
We then look at the groupoid of germs generated by the germs of $\phi_{\alpha,\beta, L}$. In other words, we now let 
\begin{equation}
\label{beta}
\beta=\beta_1,\hdots,\beta_n\quad\text{ and }\quad L=L_1,\hdots, L_n
\end{equation}
be collections of $n$ elements for any value of $n$ and define
\begin{equation}
\label{phialphabetaL}
\phi_{\alpha,\beta, L}:=\phi_{\beta_{n-1},\beta_n, L_n}\circ\hdots\circ \phi_{\alpha,\beta_1, L_1}.
\end{equation}
Here we assume by convention that both $n$ appearing in \eqref{beta} are the same, allowing repetitions if necessary. This composition is defined on some open subset of $K_{\alpha}$ that we still denote by $K_{\alpha,\beta, L}$; and it ranges in some open subset of $K_{\beta_n}$, that we denote by $K_{\bar\beta, \alpha, \bar L}$ where
\begin{equation}
\label{barbeta}
\bar\beta=(\beta_n,\hdots,\beta_1)\quad\text{ and }\quad \bar L=(L_n,\hdots, L_1).
\end{equation}
Note that
\begin{equation}
\label{compatibility}
\phi_{\beta_n, \gamma, L'}\circ\phi_{\alpha, \beta, L}\equiv \phi_{\alpha, \beta,\gamma, L, L'}
\end{equation}
where this composition is defined, and that
\begin{equation}
\label{compatibility2}
\phi_{\bar\beta,\alpha, \bar L}=(\phi_{\alpha,\beta, L})^{-1}.
\end{equation}
We define
\begin{definition}
 \label{simplegerms}
 We call {\it simple holonomy morphisms} of $G_\mathcal U$ the morphisms {\rm \eqref{phialphabetaL}}.
\end{definition}

\subsection{A first approximation of the Teichm\"uller groupoid.}
\label{firstapp}
We may encode the simple holonomy morphisms in a groupoid $K_\mathcal U$ as follows, compare with the construction of the standard holonomy groupoid in section \ref{rigidified}. It is a first approximation of the Teichm\"uller groupoid, but which does not see the automorphism groups. {\it Objects} are points of the disjoint union
\begin{equation}
\label{objectsholonomy}
\uniondisjointe_{\alpha\in A} K_{\alpha}
\end{equation}
hence encoded by couples as in \eqref{objects}. {\it Morphisms} encode germs of holonomy maps. They are defined only between a source object $(x,\alpha)$ and a target object $(y,\gamma)$ such that
\begin{equation}
\label{morphismcompatibility}
y=\phi_{\alpha,\beta, L}(x)
\end{equation}
for some collections $\beta$ (with $\beta_n=\gamma$) and $L$. We have first all identity germs, represented by a copy of \eqref{objectsholonomy} in the set of morphisms. Then, consider the maps \eqref{morphismcompatibility} for which $\beta$ - and then $L$ - has length one. They are encoded as
\begin{equation}
\label{morphismholonomy1}
\uniondisjointe_{(\alpha,\beta, L) \in B} K_{\alpha,\beta, L}.
\end{equation}
To be precise, a point $x$ in some $K_{\alpha,\beta, L}$ represents the germ at $x$ of the map $\phi_{\alpha,\beta, L}$. Here
\begin{equation}
 \label{Bdef}
 (\alpha,\beta,L)\in B\iff L\in\alpha\cap\beta\text{ and }U_{\alpha,L}\cap U_{\beta,L}\not=\emptyset.
\end{equation}
Then we represent all holonomy maps as
points  of 
\begin{equation}
\label{morphismholonomy2}
\uniondisjointe_{n\geq 0}\left (\uniondisjointe_{(\alpha,\beta, L) \in C_{n}} K_{\alpha,\beta, L}\right )
\end{equation}
for
\begin{equation}
\label{Cn}
C_n:=\left\{
\begin{aligned}
&(\alpha,\beta, L)\in A^{n+1}\times (\mathcal L)^{n}\cr
\text{such that }&(\alpha,\beta_1,L_1)\in B,\hdots, (\beta_{n-1},\beta_n,L_n)\in B
\end{aligned}
\right \}.
\end{equation}
As previously, a point $x$ in some $K_{\alpha,\beta, L}$ represents the germ at $x$ of the map $\phi_{\alpha,\beta, L}$, the case $n=0$ encoding the identity 
germs.
\vspace{5pt}\\
However, we are not finished. We must still identify identical germs. So we take the quotient of \eqref{morphismholonomy2} by the following equivalence relation
\begin{equation}
\label{equiv}
(x,\alpha,\beta,L)\sim (x',\alpha',\beta',L')\iff
\left\{
\begin{aligned}
&x=x',\ \alpha=\alpha',\ \beta_n=\beta'_{n'}\cr
&\text{and }\big (\phi_{\alpha',\beta',L'}\big )_{x'}\equiv \big (\phi_{\alpha,\beta,L}\big )_x
\end{aligned}
\right .
\end{equation}
that is if they have same source, same target, and are equal as germs.
Hence, the set of morphisms is 
\begin{equation}
\label{morphismholonomy3}
\uniondisjointe_{n\geq 0}\left (\uniondisjointe_{(\alpha,\beta, L) \in C_{n}} K_{\alpha,\beta, L}\right )\Biggm /\sim
\end{equation}

We have (cf. Proposition \ref{groupoidrigid})
\begin{proposition}
\label{groupoidII}
The groupoid $K_{\mathcal U}$ is an analytic \'etale groupoid. 
\end{proposition}
However, and contrary to the case of section \ref{rigidified} and Remark \ref{moritaequivrigid}, $K_{\mathcal U}$ and $K_{\mathcal V}$ {\it are not always Morita equivalent}. This is due to the fact that we mix holonomies of different foliations. Indeed, this is not the good holonomy groupoid to consider, because it does not take into account the groups $G_\alpha$ of the TG structure.
\begin{proof}
This is quite standard, because $K_{\mathcal U}$ is basically just a union of holonomy groupoids (cf. \cite{Ha}). 
The set of objects is obviously an analytic space by \eqref{objectsholonomy}, as well as the set defined in \eqref{morphismholonomy2}, that is the set of holonomy morphisms before taking the quotient by the equivalence relation \eqref{equiv}.  Observe that two distinct points of the same component $K_{\alpha,\beta, L}$ of \eqref{morphismholonomy3} cannot be equivalent. Therefore, the natural projection
map from \eqref{morphismholonomy2} onto \eqref{morphismholonomy3} is \'etale. Now \eqref{morphismholonomy3} is Hausdorff using the same argument as in the proof of Proposition \ref{groupoidrigid}.

For $\alpha$, $\beta$ and $L$ fixed, the source map is the inclusion
\begin{equation}
\label{sourceholonomy}
\sigma\ :\ K_{\alpha,\beta,L}\longrightarrow K_{\alpha}
\end{equation}
and the target map is given by $\phi_{\alpha,\beta, L}$, that is
\begin{equation}
\label{targetholonomy}
\begin{CD}
\tau \ :\ K_{\alpha,\beta, L}@ >\phi_{\alpha,\beta, L}>> K_{\beta_n}.
\end{CD}
\end{equation}
Composition  at the level of \eqref{morphismholonomy2} is given by
\begin{equation}
\label{compholonomy}
(x,\alpha,\beta, L)\times (y=\phi_{\alpha,\beta,L}(x), \beta_n,\gamma, L')\longmapsto (x, \alpha,\beta,\gamma, L, L'),
\end{equation}
thanks to \eqref{compatibility}. And it descends on \eqref{morphismholonomy3} as the composition of germs. This is analytic as a map from
\begin{equation}
\label{companaholonomy}
K_{\alpha,\beta,L}\cap \phi_{\alpha,\beta, L}^{-1} (K_{\beta_n, \gamma, L'})=K_{\alpha,\beta,L}\cap \phi_{\bar\beta, \alpha, \bar L} (K_{\beta_n, \gamma, L'})
\end{equation}
onto $K_{\alpha,\beta,\gamma, L, L'}$ in both cases.
\end{proof}
To finish this section, we want to clarify the relationships between $K_{\mathcal U}$ and the holonomy groupoids of the $L$-foliations. Here it is important to take special care to Remark \ref{atlases}. To avoid confusions, we will index the connected components of the objects of $G_{\mathcal U}$ by $\sqcup A_L$; and those of $K_{\mathcal U}$ by $A$. We insist on the fact that these two sets are different since we added extra indices to construct $K_{\mathcal U}$. With that difference on mind, we have immediately

\begin{proposition}
\label{autregroupoid}
Let $L\in\mathcal L$. The holonomy groupoid of the $L$-foliation is given by the full subgroupoid of $K_{\mathcal U}$ over $\sqcup_{\alpha\in A_L}K_\alpha$.
\end{proposition}

In particular, if $\mathcal L$ contains a single element, we have Morita equivalence, cf. section \ref{rigidified}.

\begin{corollary}
\label{groupoidIIbis}
Assume that $\mathcal L$ contains a single element $L$,  which is equivalent to saying that $L$ is a common complement of all $H^0(X_J,\Theta_J)$ for $J\in \mathcal F$. Then $K_{\mathcal U}$ is the holonomy groupoid of the $L$-foliation and it is independent of the covering up to Morita equivalence.
\end{corollary}

\begin{remark} 
\label{applicor}
Especially, Corollary \ref{groupoidIIbis} applies to the case where $h^0(\mathcal F)$ is zero, i.e. the automorphism group of all structures of $\mathcal F$ is discrete. But it also applies to the case of complex tori, since the continuous part of their automorphism group is given by translations and since the associated Lie algebra is independent of the complex structure (as subalgebra of the algebra of smooth vector fields). \end{remark}

\subsection{The Teichm\"uller groupoid.}
\label{Teichsub}
As in the previous subsections, we start from a regular atlas $\mathcal U$ of $V$. Here $V$ is $\I_0(a)$, or more generally any open subset of $\I$ such that $V\subset \I_0(a)$. We assume that $V$ is equal to its saturation
\begin{equation}
\label{sat}
V^{\text{sat}}:=\union_{f\in\text{Diff}^0(X)}V\cdot f.
\end{equation}
\noindent Hence, given $x$ in $V$, its complete $\text{Diff}^0(X)$-orbit is in $V$.

For simplicity, we build a new regular atlas from the first one by adding new charts as follows. Each time that $U_{\alpha, L}\cap U_{\beta,L}\not = \emptyset$, we add a covering of $U_{\alpha, L}\cap U_{\beta,L}$ by charts satisfying hypothesis \ref{squarehyp2}.  As a consequence, this new regular atlas (that we still denote $\mathcal U$) satisfies the following condition.
\begin{hypothesis}
\label{intercondition}
Every simple holonomy germ is a composition of germs of morphisms $\phi_{\alpha, \beta, L}$ with $U_{\alpha,L}\subset U_{\beta,L}$ or $U_{\beta,L}\subset U_{\alpha,L}$.
\end{hypothesis}
We call {\it adjacent} two charts $U_{\alpha,L}$ and $U_{\beta,L}$ such that $U_{\alpha,L}\subset U_{\beta,L}$ or $U_{\beta,L}\subset U_{\alpha,L}$. And we call {\it elementary holonomy germ} a holonomy germ $\phi_{\alpha, \beta, L}$ between adjacent charts. So given 
\begin{equation}
\label{ic2}
y\in K_{\alpha,\beta,L}\qquad\text{ and }\qquad y_1:=\phi_{\alpha,\beta, L}(y)\in K_{\beta, \alpha, L}
\end{equation}
with $\phi_{\alpha, \beta, L}$ elementary, \eqref{fatiso} implies that there exists a unique $\xi_1$ in $L$ such that
\begin{equation}
\label{ic3}
y_1=y\cdot e(\xi_1)
\end{equation}
Hence we have
\begin{lemma}
	\label{diff0encoding}
	Let $(\alpha,\beta,L)=(\alpha,\beta_1,\hdots,\beta_n,L_1,\hdots, L_n)$ be a path of adjacent charts. 
	To any $x$ in $K_{\alpha,\beta, L}$, is associated a canonical element in $\text{\rm Diff}^0(X)$, say $\Phi_{(x,\alpha,\beta, L)}$, such that
		\begin{equation}
		\label{diff0clave}
		x\cdot\Phi_{(x,\alpha,\beta, L)}=\phi_{\alpha,\beta, L}(x).
		\end{equation}
\end{lemma}
\begin{remark} 
	The meaning of "canonical" should be clear from the proof. 
\end{remark}

\begin{remark}
	We emphasize that \eqref{diff0clave} is a {\it pointwise identity}. Changing $x$ but keeping $(\alpha,\beta,L)$ fixed gives a different element in $\text{\rm Diff}^0(X)$, as suggested by the notations. Hence, from the one hand, \eqref{diff0clave} is far from being verified by a unique element of $\text{Diff}^0(X)$. And from the other hand, a diffeomorphism $\Phi_{(x,\alpha,\beta, L)}$ has no reason to send a neighborhood of $x$ in $K_\alpha$ onto a neighborhood of $\phi_{\alpha,\beta, L}(x)$ in $K_{\beta_n}$. 
\end{remark}

\begin{proof}
	If $\beta$ has length one, we just define $\Phi_{(x,\alpha,\beta, L)}$ as the map $e(\xi_1)$ given by \eqref{ic2} and \eqref{ic3}. Otherwise $\phi_{\alpha,\beta, L}$ has a canonical decomposition
	\eqref{phialphabetaL} into length one elements. Each element $\phi_{\beta_{i-1},\beta_i,L_i}$ of this decomposition (we take $\beta_0=\alpha$ as a convention) gives rise to an element $\Phi_{(x_{i-1},\beta_{i-1},\beta_i, L_i)}$ where $x_0=x$ and $x_i=\phi_{\beta_{i-1},\beta_i,L_i}(x_{i-1})$ since we started from a path of adjacent charts. And we just set
\begin{equation}
\label{Phi}
\Phi_{(x,\alpha,\beta, L)}:=\Phi_{(x,\alpha,\beta_1, L_1)}\circ\hdots\circ
\Phi_{(x_{n-1},\beta_{n-1},\beta_n, L_n)}
\end{equation}
which obviously satisfies \eqref{diff0clave}.
\end{proof}
\begin{remark}
	\label{contravsco}
	Be careful that composition of holonomy germs is contravariant and composition of elements $\Phi_{(x,\alpha,\beta, L)}$ is covariant.
\end{remark}

Recall the map \eqref{newGaction}. Keep in mind that we need to choose some $L\in\alpha$ to define it.
For each $\alpha\in A$, and each $L\in\alpha$, we denote by $\mathcal A_{\alpha,L}\rightrightarrows K_\alpha$ the corresponding Kuranishi stack whose geometric quotient is described in Proposition \ref{geometricquotient}.
\vspace{5pt}\\
Let us define the Teichm\"uller groupoid $T_{\mathcal U}$ as follows. {\it Objects} are points
\begin{equation}
\label{objectsteich}
(x,\alpha)\in\uniondisjointe_{\alpha\in A}K_{\alpha}
\end{equation}
exactly as for $K_{\mathcal U}$. But we will enlarge the set of morphisms to take into account the automorphism groups. We proceed as in subsection \ref{firstapp}.
\vspace{5pt}\\
First, we set $T_{\alpha, L}:=\mathcal A_{\alpha,L}$. Then, we set, for each path $(\alpha,\beta,L)$ of adjacent charts,
\begin{equation}
\label{morphismsteich1}
T_{\alpha,\beta,L}:=\mathcal A_{\alpha, L_1}\times_{\phi\circ t,s} \mathcal A_{\beta_1,L_2}\times_{\phi\circ t,s}\hdots\times_{\phi\circ t,s} \mathcal A_{\beta_{n-1},L_n}\times_{\phi\circ t,Id}K_{\beta_n}
\end{equation}
Here the $\phi$ in the fibered product $\mathcal A_{\beta_i}\times_{\phi\circ t,s}\mathcal A_{\beta_{i+1}}$ stands for $\phi_{\beta_{i},\beta_{i+1}, L_{i+1}}$. An element of $T_{\alpha, \beta, L}$ is of the form
\begin{equation}
\label{elementT}
\Big ((x,g_1),(\phi_{\alpha,\beta_1,L_1}(x\cdot g_1),g_2),(\phi_{\beta_1,\beta_2,L_2}((\phi_{\alpha,\beta_1,L_1}(x\cdot g_1))\cdot g_2,g_3),\hdots \Big )
\end{equation}
 We denote it by $(x,\alpha,\beta,L,g)$ with $g=(g_1,\hdots,g_n)$. We consider thus the space
\begin{equation}
\label{morphismsteich2}
\uniondisjointe_{n\geq 0}\left (\uniondisjointe_{(\alpha,\beta, L) \in C^{adj}_{n}} T_{\alpha,\beta, L}\right )
\end{equation}
where $C^{adj}_n$ is defined as the subset of adjacent elements of \eqref{Cn} for $n\geq 1$ and $C^{adj}_0$ is just the set of $(\alpha, L)$ with $L\in\alpha$.
 \vspace{5pt}\\
 However, as in subsection \ref{firstapp}, we still have to take the quotient of \eqref{morphismsteich2} by an appropriate equivalence relation to obtain the set of morphisms. The crucial remark to do that is to notice that there is a natural map $\Psi$ from \eqref{morphismsteich2} into $\DK$ which sends an element $(x,\alpha,\beta,L,g)$ onto
 \begin{equation}
 \label{elementDiff}
 g_1\circ \Phi_{(x\cdot g_1,\alpha,\beta_1, L_1)}\circ g_2\circ \Phi_{((\phi_{\alpha,\beta_1,L_1}(x\cdot g_1))\cdot g_2,\beta_1,\beta_2, L_2)}\circ\hdots
 \end{equation}
This allows us to identify two such morphisms with same source and target if they correspond to the same element of $\text{Diff}^0(X)$. To be precise, we define
\begin{equation}
\label{morphismsteich3}
\left .
\begin{aligned}
&(x,\alpha,\beta,L,g)\cr
&\qquad\sim\cr
&(x',\alpha',\beta',L',g')
\end{aligned}
\right \}\iff\left\{
\begin{aligned}
&x=x',\ \alpha=\alpha',\ \beta_n=\beta'_{n'}\cr
&\qquad\text{and}\cr
&\Psi(x,\alpha,\beta,L,g)=\Psi(x,\alpha,\beta',L',g')
\end{aligned}
\right .
\end{equation}
{\it Morphisms} are now defined as points 
\begin{equation}
\label{morphismsteich4}
(x,\alpha,\beta,L,g)\in\uniondisjointe_{n\geq 0}\left (\uniondisjointe_{(\alpha,\beta,L)\in C_{n}}T_{\alpha,\beta, L}\right )\Bigg /\sim.
\end{equation}

\begin{remark}
\label{subtle}
There is a subtle point here we want to emphasize. Equivalence \eqref{morphismsteich3} is an equivalence of elements in $\text{Diff}^0(X)$, whereas equivalence \eqref{equiv} is an equivalence of holonomy maps, the relation between these two type of maps being stated in Lemma \ref{diff0encoding}. In other words, \eqref{equiv} concerns the geometric orbits of $\text{Diff}^0(X)$ in $\mathcal I_0$, whereas \eqref{morphismsteich3} concerns the parametrization of the geometric orbits by $\text{Diff}^0(X)$. In particular, if an element
of $\text{Diff}^0(X)$ is an automorphism for an open neighborhood of structures in $\mathcal I_0$, then it appears as a morphism of \eqref{morphismsteich4} but not as a morphism of \eqref{morphismholonomy3}.
\end{remark}

\section{The Riemann moduli groupoid.}
\label{Riemann}

In this short section, we adapt the construction of section \ref{Teichmueller} to obtain a groupoid that describes the action of the full diffeomorphism group $\text{Diff}^+(X)$ onto $\mathcal I_0$. Fix $V$ as before. Thanks to \eqref{Riem}, we just have to add the action of the mapping class group \eqref{mcgroup} on the Teichm\"uller groupoid. To do that, we assume that $V$ is equal to its saturation
\begin{equation}
\label{sat2}
V^{\text{sat}}:=\union_{f\in\text{Diff}^+(X)}V\cdot f.
\end{equation}

To cover $V$ with Kuranishi charts, we proceed as follows. We first choose some regular covering of $V$ with Kuranishi charts satisfying Hypothesis \ref{intercondition}. Then we choose some $f_i$ in $\text{Diff}^+(X)$ for every class of $\mathcal M\mathcal C(X)$. Call $\mathcal J$ the set of indices and set $f_{\mathcal J}=(f_i)_{i\in\mathcal J}$.
We assume that $(f_i)^{-1}$ belongs to $f_{\mathcal J}$ for all $i$. But we cannot in general assume that $f_{\mathcal J}$ is stable under composition. This would imply
that we realize the mapping class group of $X$ as a subgroup of $\text{Diff}^+(X)$, which is not always possible. 
\vspace{5pt}\\
For any $U_{\alpha,L}$, we define $U_{\alpha,L}\cdot f$ and $K_\alpha\cdot f$ (well defined since $K_\alpha$ is included in $U_{\alpha,L}$), so that the sequence
\begin{equation}
\label{faction}
\begin{CD}
K_\alpha\cdot f\hookrightarrow U_{\alpha,L}\cdot f@ > (\cdot f)\circ \Xi_{\alpha,L}\circ (\cdot f^{-1})>> K_{\alpha}\cdot f
\end{CD}
\end{equation}
is a Kuranishi chart based at $J_\alpha\cdot f$.
\vspace{5pt}\\
Then we may perform the constructions of section \ref{Teichmueller}. The Riemann moduli groupoid $M_{\mathcal U}$ is now defined 
as the translation groupoid of the action of the mapping class group onto $T_{\mathcal U}$. More precisely, it is obtained as follows. We define the set of objects 
as in \eqref{objectsteich}. As for the morphisms, we start with
\begin{equation}
\label{morphismsR}
\uniondisjointe_{n\geq 0}\uniondisjointe R_{\alpha,\beta, L, I}
\end{equation}
where 
\begin{equation}
\label{R}
 R_{\alpha,\beta, L, I}=\{(x,\alpha,\beta,L,g,I)\in T_{\alpha,\beta, L}\times\mathcal J^n\}
\end{equation}
and we follow the same strategy as in section \ref{Teichmueller}.
The new map $\Psi$, say $X$ sends an element $(x,\alpha,\beta,L,g)$ onto the element $X(x,\alpha,\beta,L,g)$ defined as
\begin{equation}
\label{elementDiff+}
g_1\circ \Phi_{(x\cdot g_1,\alpha,\beta_1, L_1)}\circ f_{i_1}\circ g_2\circ \Phi_{((\phi_{\alpha,\beta_1,L_1}(x\cdot g_1))\cdot g_2,\beta_1,\beta_2, L_2)}\circ f_{i_2}\hdots\end{equation}
of $\DKp$ (compare with \eqref{elementDiff}). 

As in section \ref{Teichmueller}, we take the quotient of \eqref{morphismsR} by the equivalence relation of representing the same diffeomorphism through \eqref{elementDiff+}, cf. \eqref{morphismsteich3}. And we define the set of morphisms as this quotient. 

\section{The structure of the Teichm\"uller and the Riemann moduli stacks.}
\label{main}
In this section, building on the previous sections, we prove the main results of this paper.

\subsection{The structure of the Teichm\"uller stack.}
\label{mainTeich}
The aim of this subsection is to prove
Theorem \ref{maintheorem}. In fact, we will prove the following statement, from which Theorem \ref{maintheorem} easily follows.

\begin{theorem}
	\label{maintheorembis}
	Let $V$ be an open set of $\mathcal I$. Assume that the function $h^0$ is bounded on $V$. Then,
	the Teichm\"uller groupoid is a smooth analytic atlas of the Teichm\"uller stack $\mathcal T(X,V)$.
\end{theorem}
 
 In the general case, we have
 
\begin{corollary}
 \label{maincorollary}
 Let $V$ be an open set of $\mathcal I$. Then, the Teichm\"uller stack $\mathcal T(X,V)$ is the direct limit of Artin analytic stacks. 
\end{corollary}

\begin{proof}[Proof of Corollary \ref{maincorollary}.]
 For every nonnegative integer $a$, we define $\mathcal I(a)$ as in \eqref{I(k)}. We consider the Teichm\"uller stack $\mathcal T(X,V)$ as the direct limit of stacks
 \begin{equation}
 \label{directlimit}
 \mathcal T(X,V\cap \I(0))\hookrightarrow \hdots\hookrightarrow \mathcal T(X,V\cap \I(a))\hookrightarrow \hdots
\end{equation}
Applying then Theorem \ref{maintheorem} replacing $V$ with $V\cap \mathcal I(a)$ for every $a$ yields the result.
\end{proof}

The manifold $\mathbb S^2\times \mathbb S^2$ gives such an example, cf. Example \ref{Hirzebruchbis}.\vspace{5pt}\\
%
%
We begin with showing that the set of morphisms of $T_{\mathcal{U}}$ completely describes the action of $\text{Diff}^0(X)$.

\begin{lemma}
\label{isoteich}
We have:
\begin{enumerate} 
\item[(i)] Let $x\in K_\alpha$ an object. Then the set of $x$-isomorphisms is $\text{\rm Aut}^1(X_x)$.
\item[(ii)] Let $x\in K_\alpha$ and $y\in K_\beta$. Then the set of morphisms from $x$ to $y$ is the set 
\begin{equation}
\label{morphismsxy}
\{f\in\text{\rm Diff}^0(X)\mid x\cdot f=y\}.
\end{equation}
\end{enumerate}
\end{lemma}

\begin{proof}
(ii) Let $x\in K_{\alpha}$ be an object. It is only connected through a morphism to a point $y$ in some $K_{\tilde\alpha}$ which belongs to the same orbit of $\text{Diff}^0(X)$.
Let now $y\in K_{\tilde\alpha}$ such that 
\begin{equation}
\label{itii}
y=x\cdot f
\end{equation}
for some $f$ in $\text{Diff}^0(X)$. Choose also an isotopy
\begin{equation}
\label{itisotopy}
y_t=x\cdot f_t
\end{equation}
from $x$ to $y$. To this isotopy is associated a sequence $t_0=0<t_1<...<t_{n-1}<t_n=1$ and a path of adjacent charts $(\alpha=\beta_0,\beta,L)$ (with $\tilde\alpha=\beta_n$) such that $y_t$ belongs to $U_{\beta_i,L_{i+1}}$ for $t_i\leq t\leq t_{i+1}$.  Choose $t_{i-1}<T_i<t_{i}$ for all $0<i<n$ such that 
\begin{equation}
y_{T_i}\in U_{\beta_{i-1},L_{i}}\cap U_{\beta_{i},L_{i+1}}
\end{equation}
 We may modify locally $f$ around each time $T_i$ in such a way that $y_{T_i}$ belongs to $K_{\beta_i}$ but still $y_t$ belongs to $U_{\beta_i,L_{i+1}}$ for $t_i\leq t\leq t_{i+1}$. Indeed, setting $\xi=\Upsilon_{\beta_i,L_{i+1}}(y_{T_i})$, we have by definition that $y_{T_i}\cdot e(\xi)$ belongs to $K_{\beta_i}$. 
 It is thus enough to take a bump function $b_i$ with support around $T_i$ such that $b(T_i)=1$ and $y_t\cdot e(b_i(t)\xi)$ stays in $U_{\beta_i,L_{i+1}}$ for $t_i\leq t\leq t_{i+1}$.

Decompose now $f$ as $f^1\circ\hdots\circ f^n$ with $f^1:=f_{T_1}$, then $f^2=f_{T_1}^{-1}\circ f_{T_2}$ and so on.

We claim that $f^1$ belongs to set of morphisms of the Teichm\"uller groupoid between $y_{T_0}=y_0=x\in K_\alpha$ and $y_{T_1}\in K_{\beta_1}$. Indeed, setting
\begin{equation}
\label{ut}
u_t:=f_t\circ e(\Upsilon_{\beta_0,L_1}(y_t))
	\qquad 0\leq t\leq T_1
\end{equation} 
we see that
\begin{equation}
\label{zt}
z_{t}=y_{T_0}\cdot u(t)\in K_{\beta_0}
\qquad \text{ for all }0\leq t\leq T_1
\end{equation}
Now, we deduce easily from \eqref{zt} and \eqref{ut} that $u_{T_1}$ belongs to set of morphisms of the Teichm\"uller groupoid. Indeed, we can write $u_{T_1}$ as a finite composition of elements of $D_{\alpha,L}$ sending a point of $K_\alpha$ to another point of $K_\alpha$, so $u_{T_1}$ is a morphism of $\mathcal A_{\alpha,L}$. 

Moreover, $e(\Upsilon_{\beta_0,L_1}(y_{T_1}))$ is also an element of $D_{\alpha,L}$, hence $f_1$ being a composition of elements of $D_{\alpha,L}$ is also a morphism of $\mathcal A_{\alpha,L}$.
The same line of arguments proves that $f_i$ is a morphism of $\mathcal A_{\beta_i,L_{i+1}}$ for all i, hence, by composition, we are done.
\vspace{5pt}\\
 (i) Just apply (ii) to the case $x=y$.
\end{proof}

\begin{remark}
\label{Aut10}
Notice from the proof of Lemma \ref{morphismsxy} that an element of $\text{Aut}^1(X)$ which is not in $\text{Aut}^0(X)$ decomposes into a non trivial combination of holonomy maps and automorphisms of $\text{Aut}^0(X)$.
\end{remark}

We are now in position to prove Theorem \ref{maintheorembis}.

\begin{proof}[Proof of main Theorem \ref{maintheorembis}]
Let us start proving that the Teichm\"uller group\-oid is smooth analytic. First, the set of objects is a countable union of analytic spaces by \eqref{objectsteich}. To prove that the set of morphisms is an analytic space and the the source map a smooth morphism, we proceed as in the proof of Proposition \ref{analyticKurstack}. Recall that the map $\Psi$ gives a continuous injection from the set of morphisms \eqref{morphismsteich4} to $\DK$. Hence it is Hausdorff. Besides, the map $\Psi$ realizes \eqref{morphismsteich4} as an an analytic subspace of $\DK$. Thanks to Lemma \ref{isoteich}, it is the set of elements $(x,f)$ in $\DK$ such that $x\cdot f$ belongs to the disjoint union of the $K_\alpha$. 

Then, if $(x,\Psi(x,\alpha,\beta, L,g))$ is a morphism, every morphism $(x',F')$ close to it can be written as $(x',\Psi(x,\alpha,\beta, L,g)\circ h\circ e(\xi))$, with $h\circ e(\xi)\in \mathcal D_{\beta_n,L_n}$ and  
\begin{equation}
\label{xichart}
\xi=\Upsilon_{\beta_n,L_n}(x'\cdot (\Psi(x,\alpha,\beta, L,g)\circ h))
\end{equation}
so we have local isomorphisms (compare with \eqref{KuranishiStackChart})
\begin{equation}
\label{TeichStackChart}
(x',F')\longmapsto (x',h)\in K_\alpha\times \text{Aut}^0(X_{\beta_n})
\end{equation}
In charts \eqref{TeichStackChart}, the source map is just the projection $(x',F')\mapsto x'$
so is analytic and a smooth morphism.
\vspace{5pt}\\
Multiplication, resp. inverse is given by composition, resp. inverse, of diffeomorphisms in $\DK$. The target map is given by action of the diffeomorphisms. Recall that the action is analytic and that composition, resp. inverse, when restricted to finite dimensional analytic subspaces containing only $C^\infty$ elements, are also analytic, cf. Proposition \ref{analyticKurstack} and its proof. The anchor map is obviously analytic. \vspace{5pt}\\
We prove now that the stackification of the Teichm\"uller groupoid is $\mathcal T(X,V)$.
Let $\mathcal U$ be a regular atlas of $V$. 
We assume \eqref{intercondition}.
An object over $S$ in the stackification of $T_{\mathcal U}$ is given by an open covering $(S_a)$ of $S$,  a collection of maps
\begin{equation}
\label{fa}
f_a \ :\ S_a\longrightarrow K_{\alpha}
\end{equation}
($\alpha$ depends on $a$) and a collection of gluings
\begin{equation}
\label{hab}
h_{ab}=(f_{ab},g_{ab})\ : S_a\cap S_b \longrightarrow \Psi(T_{\alpha, \beta, L})\subset \DK
\end{equation}
satisfying a compatibility condition as well as the usual cocycle condition. More precisely, the compatibility condition is that,
given $x$ in $S_a\cap S_b$, we have
\begin{equation}
\label{sigma}
\sigma (h_{ab}(x))=f_{ab}(x)=f_a(x)
\end{equation}
and 
\begin{equation}
\label{tau}
\tau(h_{ab}(x))=( f_{ab}(x))\cdot g_{ab}(x)=f_b(x).
\end{equation}
We will show that this is exactly the data we need to construct a $(X,V)$-family $\mathcal X$. Set
\begin{equation}
\label{XV1}
\mathcal K_{\alpha}:=(K_{\alpha}\times X, \mathcal J_\alpha)
\end{equation}
where the operator $\mathcal J_\alpha$ along the fiber $\{J\}\times X$ is tautologically defined as $J$. We use the $C^\infty$-marking given by the trivialization $K_\alpha\times X$. This defines a $(X, V)$-family over $S_a$, cf. \cite{Ku3}. 
\vspace{5pt}\\
The main point is that $g_{ab}$ lifts canonically to an isomorphism between the restriction of $\mathcal K_{\alpha}$ over $f_a(S_a\cap S_b)$ and the restriction of $\mathcal K_{\beta}$ over $f_b(S_a\cap S_b)$. Define the canonical lifting of \eqref{tau} as
\begin{equation}
\label{clifting}
X_{ab}(x,y):=\Big (( f_{ab}(x))\cdot g_{ab}(x),(g_{ab}(x))(y)\Big ) 
\end{equation}
for
\begin{equation}
\label{liftingcond}
x\in S_a\cap S_b\quad\text{ and }\quad y\in X.
\end{equation}
Observe that the cocycle condition just means that the maps
$(g_{ab}(x))$
verify the cocycle condition in $\text{Diff}^0(X)$. Hence the $C^\infty$-markings coincide on the intersections. Now, define $\mathcal X$ as
\begin{equation}
\label{defX}
\mathcal X=\uniondisjointe_a f_a^*(\mathcal K_{\alpha})/\sim
\end{equation}
where $\sim$ is the equivalence relation
\begin{equation}
\label{XVequiv}
(x,y,a)\sim (x',y',b)\iff (x',y')=X_{ab}(f_a(x),y).
\end{equation}
This defines a $(X,V)$-family thanks to the cocycle condition.
\vspace{5pt}\\
Hence, every locally trivial torsor associated to $T_{\mathcal U}$ is a $(X,V)$-family. 
\vspace{5pt}\\
Let $S\in\mathfrak S$ and $S'\in\mathfrak S$. Let $g : S\to S'$ be a morphism. Let $(f_a,S_a, h_{ab})$, respectively $(f'_{a'}, S'_{a'}, h'_{a'b'})$ be an object over $S$, respectively $S'$ (we use \eqref{fa}, \eqref{hab} and so on). A morphism between them and over $g$ is given by a collection of maps $F_{aa'}$ from $S_a$ to the set of morphisms of $T_{\mathcal U}$ such that
\begin{enumerate}
	\item[(i)] For all $x\in S_a$, we have $\sigma (F_{aa'}(x))=f_a(x)$ and $\tau (F_{aa'}(x))=f'_{a'}\circ g(x)$.
	\item[(ii)] $F_{bb'}\circ h_{ab}=h'_{a'b'}\circ F_{aa'}$.
\end{enumerate}
It is straightforward, although awkward, to check that (i) shows that $F_{aa'}$ induces local cartesian diagrams
\begin{equation}
\label{XVCD}
\begin{CD}
f_a^*\mathcal K_\alpha @ >>> (f'_\alpha)^*\mathcal K_{\alpha'}\cr
@ VVV @ VVV\cr
S_a@ >>g> S'_{a'}
\end{CD}
\end{equation}
that is local morphisms between the families associated to the descent data; and that (ii) implies that these local morphisms commute with the gluing \eqref{XVequiv}, hence define a global morphism of $(X,V)$-families.
\vspace{5pt}\\
All this shows the existence of a functor over $\mathfrak S$ from the stackification of $T_{\mathcal U}$ to $\mathcal T(X, V)$. But Kuranishi's Theorem shows that any $(X,V)$-family is locally isomorphic to a pull-back family $f_a^*\mathcal K_\alpha$. Hence we may choose a covering of the base and a collection of maps $f_a$  as in \eqref{fa}, with associated gluing maps \eqref{hab} satisfying \eqref{sigma} and \eqref{tau} so that it is isomorphic to some family \eqref{defX}. Hence this functor is essentially surjective.
\vspace{5pt}\\
Moreover, because of Lemma \ref{morphismsxyR}, morphisms between two objects of the stackification of $T_{\mathcal U}$ coincide with morphisms between them as objects of $\mathcal T(X,V)$. Therefore the functor is fully faithful and the two stacks are indeed isomorphic.
This finishes the proof.
\end{proof}
We notice the following

\begin{corollary}
 \label{etale?}
 The Teichm\"uller groupoid is an \'etale analytic presentation of the Teichm\"uller stack $\mathcal T(X,V)$ if and only if the function $h^0$ is identically zero on $V$.
\end{corollary}

\begin{proof}
Use Theorem \ref{maintheorembis} and the fact that the isotropy group of a point $J$ is $\text{Aut}^1(X_J)$ by Lemma  \ref{isoteich}. 
\end{proof}

Moreover, if all structures in $V$ are rigidified, then the \'etale Teichm\"uller groupoid coincides with that of Definition \ref{teichrigid}.

\begin{remark}
\label{Catanese}
It is important to compare the local structure of the Teichm\"uller stack at some point $J$ with its Kuranishi space $K_J$, or better with its Kuranishi stack $\mathcal A_J\rightrightarrows K_J$. The rigidified case is of special interest and amounts to asking if the Teichm\"uller stack of $X$ is locally isomorphic at $J$ to the analytic space $K_J$, cf \cite{Catsurvey}.
\vspace{5pt}\\
Catanese shows in \cite[Theorem 45]{Catsurvey}, that, for a minimal surface $S$ of general type, if $\text{Aut}(S)$ is a trivial group, or
if  $S$ is rigidified with ample canonical bundle, then the Teichm\"uller space is locally homemorphic to the Kuranishi space. He also shows in \cite[Proposition 15]{Catsurvey} that the same result holds for K\"ahler manifolds with trivial canonical bundle. This is used by Verbitsky in \cite{Verbitsky}, see Example \ref{HK}.
\vspace{5pt}\\
This question is equivalent to asking if there can be non trivial simple holonomy morphisms. In particular, when {\it all} the structures of a connected component $\mathcal I_0$ are rigidified, a positive answer means that the holonomy groupoid of the $\text{Diff}^0(X)$-foliation of $\mathcal I_0$ is trivial, hence that the foliation itself is trivial.\vspace{5pt}\\
This seems however too much to expect in general and suggests the following
\begin{problem}
 \label{holonomyproblem}
 Find a compact $C^\infty$ manifold $X$ with a connected component $\mathcal I_0$ of rigidified structures and with a non-trivial Teichm\"uller groupoid.
\end{problem}
To begin with, it would be very interesting to have an example of an oriented smooth manifold $X$ such that $\mathcal T(X)$ is the leaf space of an irrational foliation of a complex torus.
\end{remark}

\subsection{The structure of the Riemann moduli stack.}
\label{mainRiemann}
Analogously, we will prove Theorem \ref{mainR}. It is obtained as an easy consequence of the more precise
\begin{theorem}
	\label{mainRbis}
	Let $V$ be an open subset of $\mathcal I$. Assume that the function $h^0$ is bounded on $V$. Then, the Riemann moduli groupoid is a smooth analytic atlas of the Riemann moduli stack $\mathcal M(X,V)$.
\end{theorem}
In the general case, 
 
\begin{corollary}
 \label{maincorollaryR}
 Let $V$ be an open subset of $\mathcal I$. Then, the Riemann moduli stack $\mathcal M(X,V)$ is the direct limit of Artin analytic stacks. 
\end{corollary}

The proof of Corollary \ref{maincorollaryR} is similar to that of Corollary \ref{maincorollary}. The proof of Theorem \ref{mainRbis} follows that of Theorem \ref{maintheorembis}.
As in the previous section, we first notice that

\begin{lemma}
\label{morphismsxyR}
Pick $x$ and $y$ in the set of objects. Then, the set of morphisms  joining $x$ to $y$ is
\begin{equation}
\label{morxyR}
\{f\in\text{\rm Diff}^+(X)\mid y=x\cdot f\}.
\end{equation}
\end{lemma}

\begin{proof}
Let $f$ belong to \eqref{morxyR}. Then, there exists $i\in\mathcal J$ such that $f\circ f_i$ belongs to $\text{Diff}^0(X)$. By Lemma \ref{morphismsxy}, we know that
$f\circ f_i$ belongs to the set of morphisms joining $x$ to $y\cdot f_i$. Hence $f=f\circ f_i\circ f_i^{-1}$ belongs to the set of morphisms joining
$x$ to $y$.
\end{proof}
 
%
%
%
%

Finally, the proof that the stackification of the Riemann groupoid is isomorphic to $\mathcal M(X,V)$ is completely analogous to the corresponding proof for the Teichm\"uller groupoid. We only have to consider in \eqref{hab} that $g_{ab}$ has two components $(g^1_{ab},g^2_{ab})$, the second one being in $\Gamma$ and to add the right action of $g^2_{ab}$ in \eqref{tau}, \eqref{clifting}.\vspace{5pt}\\
Notice the obvious Corollary

\begin{corollary}
\label{Lindep}
We have:
\begin{enumerate}
\item[(i)] The groupoid $M_{\mathcal U}$ is independent of $\mathcal L$ up to analytic Morita equivalence.
\item[(ii)] The groupoid $T_{\mathcal U}$ is independent of $\mathcal L$ up to analytic Morita equivalence.
\end{enumerate}
\end{corollary}

\begin{proof}
Since both stackifications are completely independent of $\mathcal L$ by Theorems \ref{mainRbis} and \ref{maintheorembis}, we have directly the results. 
\end{proof}

\begin{remark}
 \label{RvsT}
 In the classical case of Riemann surfaces, the Teichm\"uller space is nicer than the Riemann moduli space, since the first one is a manifold whereas the second one is an orbifold. There is no such difference between the Teichm\"uller stack and the Riemann moduli stack. Both have similar structures of Artin analytic stacks. However, the Teichm\"uller groupoid has a much more natural geometric interpretation as the holonomy groupoid of the TG foliated structure of $\mathcal I$. The Riemann moduli stack is built from this holonomy groupoid and from the action of the mapping class group. Hence, for quite different reasons than for surfaces, the Teichm\"uller stack is nicer than the Riemann moduli stack.
\end{remark}

\section{Examples.}
\label{examples}
\begin{example}{\bf Tori.}
\label{tori}
Consider firstly the one-dimensional case. So let $X$ be $\mathbb S^1\times\mathbb S^1$. Then $\mathcal I$ is connected and, as geometric quotients, $\mathcal T(X)$ is the upper half plane $\mathbb H$, and $\mathcal M(X)$ is the orbifold obtained as the quotient of $\mathbb H$ by the classical action \eqref{sl2action} of $\text{SL}_2(\mathbb Z)$.
\vspace{5pt}\\
However, these are not  the Teichm\"uller and Riemann stacks of $X$, but of $X$ with a fixed point, that is they are the Teichm\"uller and Riemann stacks of $X$ for structures of elliptic curves.
\vspace{5pt}\\
To describe $\mathcal T(X)$ and $\mathcal M(X)$ as stacks, we must incorporate the action of the translations. This can be done as follows. Consider the quotient $\mathcal X$ of $\mathbb C\times\mathbb H$ by the group generated by
\begin{equation}
\label{vtorigen}
(z,\tau)\longmapsto (z+1,\tau)\quad\text{ and }\quad (z,\tau)\longmapsto (z+\tau,\tau)
\end{equation}
Then
\begin{equation}
\label{vtori}
[z,\tau]\in\mathcal X\longmapsto \pi [z,\tau]:=\tau\in\mathbb H
\end{equation}
is a universal family for all $1$-dimensional tori, cf. \cite{MK}, pp.18-19. Then, we may take as Teichm\"uller groupoid, the groupoid
\begin{equation}
\label{toriTgroupoid}
\mathcal T(X)=\left [\mathcal X\rightrightarrows \mathbb H\right]
\end{equation}
where the source and target maps are both equal to the projection map $\pi$ of \eqref{vtori} and where composition is just addition. This must be understood as follows. The common fibers at a point $\tau$ is the elliptic curve $\mathbb E_\tau$ which must be thought of as the translation group of  $\pi^{-1}(\tau)$. Observe that even if we are considering tori, the family $\mathcal X$ has a natural section, namely the image of $\{0\}\times \mathbb H$ through \eqref{vtorigen}, allowing a natural identification between $\pi^{-1}(\tau)$ and its translation group. The fact that the source and target maps coincide reflects the stability of the translation groups as explained in Remark \ref{applicor}.
\vspace{5pt}\\
To describe the Riemann groupoid, we now just have to add the $\text{SL}_2(\mathbb Z)$ action. Given
\begin{equation}
\label{matrixA}
A=\begin{pmatrix}
p &q\cr
r &s
\end{pmatrix}
\end{equation}
an element of $\text{SL}_2(\mathbb Z)$, recall that
\begin{equation}
\label{sl2action}
A\cdot \tau=\dfrac{p\tau+q}{r\tau+s}.
\end{equation}
Just set now
\begin{equation}
\label{toriMgroupoid}
\mathcal M(X)=\left[\text{SL}_2(\mathbb Z)\times\mathcal X\rightrightarrows \mathbb H\right ]
\end{equation}
where the source map is $\pi$, the target map is given by the $\text{SL}_2(\mathbb Z)$ action, and composition follows the rule
\begin{equation}
\label{comptori}
(B,[b]_{A\cdot \tau},A\cdot \tau)\circ (A,[a]_\tau,\tau) =(BA,[a+b(r\tau+s)]_\tau, \tau)
\end{equation}
for $A$ defined in \eqref{matrixA} and $[z]_\tau$ meaning the class of $z\in\mathbb C$ modulo $\mathbb Z\oplus\mathbb Z\tau$.
\vspace{5pt}\\
Let us treat now the higher dimensional case. It follows exactly the same pattern. A universal family is described in \cite[\S 5.2]{KodairaBook}. One replaces $\mathbb H$ with
\begin{equation}
\label{hdH}
\mathcal H_n:=\{T\in \text{M}_n(\mathbb C)\mid \det \text{Im }T >0\}
\end{equation}
and one takes the quotient $\mathcal X_n$ of $\mathcal H_n\times \mathbb C^n$ by the action generated by
\begin{equation}
\label{hdtoriTeich}
(T,z)\longmapsto (T, z+e_i)\quad\text{ and }\quad(T,z)\longmapsto (T, z+T_i)
\end{equation}
where $(e_i)$ is the canonical basis of $\mathbb C^n$ and $(T_i)$ the rows of $T$. Then the Teichm\"uller stack can be presented as
\begin{equation}
\label{hdtoriTgroupoid}
\mathcal T(X)=\left [\mathcal X_n\rightrightarrows \mathcal H_n\right]
\end{equation}
where the source and target maps are both equal to the projection map and where composition is just addition. Finally, given
\begin{equation}
\label{matrixAn}
A=\begin{pmatrix}
P &Q\cr
R &S
\end{pmatrix}
\end{equation}
an element of $\text{SL}_{2n}(\mathbb Z)$ decomposed into blocks of size $n\times n$, recall that
\begin{equation}
\label{sl2naction}
A\cdot T=(PT+Q)(RT+S)^{-1}
\end{equation}
is the action of $\text{SL}_{2n}(\mathbb Z)$ onto $\mathcal H_n$ identifying biholomorphic complex tori.
Just set now
\begin{equation}
\label{hdtoriMgroupoid}
\mathcal M(X)=\left[\text{SL}_{2n}(\mathbb Z)\times\mathcal X_n\rightrightarrows \mathcal H_n\right ]
\end{equation}
where the source map is the projection, the target map is \eqref{sl2naction}, and composition follows the rule
\begin{equation}
\label{hdcomptori}
(B,[b]_{A\cdot T},A\cdot T)\circ (A,[a]_T, T)=(BA,[a+b(RT+S)]_T, T)
\end{equation}
The geometric quotients are $\mathcal H_n$ as Teichm\"uller space and the quotient of $\mathcal H_n$ by the action \eqref{sl2naction} as Riemann space. Notice however that this is far from being an orbifold, cf. \cite[\S 5.2]{KodairaBook} and \cite{VerbitskyErgodic}. 
\end{example}
\begin{example}{\bf Hyperk\"ahler manifolds.}
\label{HK}
We make the connection between our general results and the beautiful description of the Teichm\"uller space for simple hyperk\"ahler manifolds in \cite{Verbitsky}, to which we refer for further details. Let $X$ be any oriented smooth compact manifold admitting hyperk\"ahler structures. We restrict $\mathcal I$ to complex structures {\it of hyperk\"ahler type}. It has a finite number of connected components. It follows from Proposition 15 of \cite{Catsurvey} and the injectivity of the local period map that $\mathcal T(X)$ coincide locally with the Kuranishi space. Moreover, we consider only simple hyperk\"ahler structures, that is simply connected ones. This implies that the first cohomology group with values in the structure sheaf is zero. So is the group of global $(n-1)$ holomorphic forms by Serre duality. Hence, by pairing, these simple hyperk\"ahler manifolds do not admit any non zero holomorphic vector field.
\vspace{5pt}\\
In our setting, this means that 
\begin{enumerate}
\item[(i)] $\mathcal T(X)$ is \'etale, see Corollary \ref{etale?}, and, taking into account Remark  \ref{Aut10bis}, coincides with the holonomy groupoid constructed in subsection \ref{firstapp} up to a finite morphism\footnote{We do not know if any simple hyperk\"ahler manifold is rigidified. In case it is, recall that $\mathcal T(X)$ coincides with the holonomy groupoid.}.
\item[(ii)] There is no non trivial holonomy germ.
\end{enumerate}
In other words, $\mathcal T(X)$ is, up to a finite morphism, the leaf space of a single $L$-foliation, and this foliation has no holonomy, so this leaf space is locally Hausdorff and coincides locally at each point with the Kuranishi space of this point. By Bogomolov-Tian-Todorov Theorem, the Kuranishi space is a manifold, so the leaf space is locally a complex manifold.
\vspace{5pt}\\
This is however not enough to imply Hausdorffness; but it forces the inseparable points to lie on a subset of measure zero. And it gives $\mathcal T(X)$ the structure of a non-Hausdorff complex manifold.
\vspace{5pt}\\
In this particular case, Verbitsky shows in Theorem 1.15 that the inseparibility condition is an equivalence relation and that the quotient of $\mathcal T(X)$ by this equivalence relation is a Hausdorff complex manifold that he calls the birational Teichm\"uller space (taking into account that, following a result by Huybrechts, inseparable points correspond to birational hyperk\"ahler manifolds). 
\vspace{5pt}\\
Finally, the action of the mapping class group on $\mathcal T(X)$ can be very complicated, see \cite{VerbitskyErgodic}.
\end{example}
\begin{example}{\bf Hopf surfaces.}
\label{Hopfbis}
We go back to the Hopf surfaces of Example \ref{Hopf}. We assume the reader to be acquainted with deformation theory of primary Hopf surfaces as detailed in \cite{We}. We consider a connected component $\mathcal I_0$ of $\mathcal I$. Looking at the $f$-homotopy graph of Figure \ref{Hopffig}, we see that it is enough to use Kuranishi spaces of type IV and type III Hopf surfaces. It follows from \cite{We} and Lemma \ref{Hopfcc} that
\begin{enumerate}
\item[(i)] We have $\mathcal T(X,\mathcal I_0)=\mathcal M(X)$.
\item [(ii)] We have $\text{Aut}(X_J)=\text{Aut}^1(X_J)=\text{Aut}^0(X_J)$ for all structures $J$.
\end{enumerate}
All type IV can be described as a single Kuranishi family constructed as follows, cf. \cite{We} and \cite{Dabrowski}. Define
\begin{equation}
\label{typeIV}
U:=\left\{
A\in \text{GL}_2(\mathbb C)
\quad\text{such that}\quad\left\{\begin{aligned}&(i)\ 0<\vert\lambda_1\vert \leq \vert\lambda_2\vert<1\cr
&(ii)\ \lambda_1 =\lambda_2^p\Longrightarrow p=1
\end{aligned}\right .
\right\}
\end{equation}
for $\lambda_1$ and $\lambda_2$ the eigenvalues of $A$.
Set
\begin{equation}
\label{typeIVfamily}
\mathcal X_U:=\big (\mathbb C^2\setminus\{(0,0)\}\times U\big )\Big /\big\langle (Z,A)\mapsto (A\cdot Z, A)\big\rangle
 \end{equation}
Then $\mathcal X_U\to U$ is a versal family for every Hopf surface of type IV, which is moreover complete for every surface of type IIb and of type IIc. Let $p>1$ and define
\begin{equation}
\label{typeIIIp}
V_p:=\left\{
(\lambda_1,\lambda_2,\alpha)\in\mathbb C^3
\quad\text{with}\quad\left\{\begin{aligned}&(i)\ 0<\vert\lambda_1\vert < \vert\lambda_2\vert<1\cr
&(ii)\ \lambda_1 =\lambda_2^q\Longrightarrow q=p
\end{aligned}\right .
\right\}
\end{equation}
Set
\begin{equation}
\label{typeIIIpfamily}
\mathcal X_{V_p}:=\big (\mathbb C^2\setminus\{(0,0)\}\times V_p\big )\Big /\big\langle (z,w,A)\mapsto (\lambda_1z+\alpha w^p,\lambda_2w, A)\big\rangle
\end{equation}
for $A=(\lambda_1,\lambda_2,\alpha)$. Then $\mathcal X_{V_p}\to V_p$ is a versal family for every Hopf surface of type III with weight $p$, which is moreover complete for every surface of type IIa with weight $p$ and of type IIc. Incorporating the automorphism groups, we define
\begin{equation}
\label{TIV}
\mathcal T_{IV}:=\big (\text{GL}_2(\mathbb C)\times U\big )\big / \langle (M,A)\mapsto (MA, A)\rangle
\end{equation} 
and consider the groupoid
\begin{equation}
\label{groupoidIV}
\mathcal T_{IV}\rightrightarrows U
\end{equation}
where the source map is the projection onto the second factor of \eqref{TIV}; the target map is the conjugation of the second factor by the first one; and the composition follows the rule
\begin{equation}
\label{compIV}
[N,MAM^{-1}]\circ[M,A]=[NM,A]
\end{equation} 
Then \eqref{groupoidIV} is a Teichm\"uller groupoid for a neighborhood of the $f$-homotopy class IV including all type IV, IIb and IIc Hopf surfaces. In the same way, let  
\begin{equation}
\label{Gp}
G_p=\{(z,w)\mapsto (az+bw^p, dw)\mid ad\not =0\}
\end{equation}
and define
\begin{equation}
\label{TIIIp}
\mathcal T_{IIIp}:=\big (G_p\times V_p\big )\big / \langle (M,A)\mapsto (MA, A)\rangle
\end{equation} 
with the convention that, given $A=(\lambda_1,\lambda_2,\alpha)$ and given $M$ with coefficients $(a,b,d)$, then
\begin{equation}
\label{AM}
MA:=\big ((z,w)\longmapsto (a\lambda_1 z+(a\alpha + b\lambda_2^p)w^p,d\lambda_2 w)\big )
\end{equation}
Consider the groupoid
\begin{equation}
\label{groupoidIIIp}
\mathcal T_{IIIp}\rightrightarrows V_p
\end{equation}
where the source map is the projection onto the second factor of \eqref{TIIIp}; the target map is the conjugation of the second factor by the first one using \eqref{AM}; and the composition is given by composition in $G_p$. 
Then \eqref{groupoidIIIp} is a Teichm\"uller groupoid for a neighborhood of the $f$-homotopy class III of weight $p$ including all type III of weight $p$, IIb of weight $p$ and IIc Hopf surfaces.
\vspace{5pt}\\
To finish with, we consider the disjoint union of groupoid \eqref{groupoidIV} and of groupoids \eqref{groupoidIIIp} for all $p>1$. We need to add the holonomy morphisms between these groupoids. 
In this case, it is not even necessary to fat the spaces, since we have natural identifications
\begin{equation}
\label{Hopfiden}
\big ((\lambda_1,\lambda_2,0)\in V_p \text{ such that } \lambda_1\not = \lambda_2^p\big )\sim
\begin{pmatrix}
\lambda_1 &0\cr
0 &\lambda_2
\end{pmatrix}
\in U
\qquad (p>1)
\end{equation}
So we take as presentation of $\mathcal T_{\mathcal I_0} (X)$ the groupoid whose objects are
\begin{equation}
\label{Hopfobjects}
U\uniondisjointe_{p>1} V_p
\end{equation}
and whose morphisms are generated by morphisms of \eqref{groupoidIV} and \eqref{groupoidIIIp} for all $p>1$ from the one hand, and by identifications \eqref{Hopfiden} from the other hand. To be more precise, set
\begin{equation}
\label{morphismssuppHopf}
W_p:=\{(\lambda_1,\lambda_2,0)\in V_p \text{ such that } \lambda_1\not = \lambda_2^p\}\qquad (p>1)
\end{equation}
and define the supplementary set of morphisms as 
\begin{equation}
\label{TIIIpIV}
\mathcal T_{IIIpIV}:=\left (\left\{M=
\begin{pmatrix}
a &0\cr
0 &d
\end{pmatrix}
\right \}
\times W_p\right )\Big / \langle (M,A)\mapsto (MA, A)\rangle
\end{equation} 
with source map being the second projection and target map being conjugation of the second factor by the first one composed with identification \eqref{Hopfiden}.
Hence the set of morphisms is generated from
\begin{equation}
\label{morphismsHopf}
\mathcal T_{IV}\uniondisjointe_{p>1}\mathcal T_{IIIp}\uniondisjointe_{p>1}\mathcal T_{IIIpIV}
\end{equation}
using the process explained in section \ref{Teichmueller}.
Recall that $\mathcal T(X,\mathcal I_0)$ is equal to $\mathcal M(X)$, hence this gives also a presentation of $\mathcal M(X)$.
\vspace{5pt}\\
Finally, we give a model for the geometric quotient of $\mathcal T(X,\mathcal I_0)$. Consider the map
\begin{equation}
\label{dettrace}
A\in\text{GL}_2(\mathbb C)\longmapsto \phi(A):=(\text{det }A,\text{Tr }A)\in\mathbb C^*\times\mathbb C
\end{equation}
then $\phi(U)$ coincides with the quotient space of $U$ by the conjugation action of $\text{GL}_2(\mathbb C)$ except for matrices with a single eigenvalue.
\vspace{5pt}\\
From this, it is easy to check that the geometric quotient can be constructed as follows.
\begin{itemize}
\item Start with the domain
\begin{equation}
\label{domain}
D=\phi(\{A\in GL_2(\mathbb C)\mid 0<\vert \lambda_1\vert\leq \vert\lambda_2\vert <1\})\subset\mathbb C^*\times\mathbb C
\end{equation}
that is with the image by $\phi$ of the set of invertible matrices with both eigenvalues having modulus strictly less than one.
\item Double asymmetrically the points of the analytic subspace
 \begin{equation}
\label{gluingIV}
\{(1/4t^2,t)\mid 0<\vert t\vert <2\}\subset D
\end{equation}
making $D$ non-Hausdorff along \eqref{gluingIV}. This encodes the fact that above such a point (for $\phi$), there are two distinct $\text{GL}_2(\mathbb C)$-orbits and not a single one. Note that these points correspond to type IV Hopf surfaces.
\item For each value of $p>1$, double asymmetrically the points of the analytic subspace
\begin{equation}
\label{gluingIIIp}
\{(t^{p+1},t+t^p)\mid 0<\vert t\vert <1\}\subset D
\end{equation}
making $D$ non-Hausdorff along \eqref{gluingIIIp}. This encodes the jumping phenomenon of type III Hopf surfaces of weight $p$.
\end{itemize}
By {\it doubling asymmetrically the points along some subset} $C$, we mean that we replace the subset $C$ by $C\sqcup C$ with the following topology. The second component is endowed with the topology of $C\subset D$. But given any point $P$ in the first component of $C\sqcup C$, then every neighborhood of $P$ contains also the corresponding point $Q$ in the other component. Hence, $P$ and $Q$ are not separated, however they do not play the same role and the situation is not symmetric.

\begin{remark}
 \label{Cstar}
 Let $\mathbb C^*$ act by homotheties onto $\mathbb C$. Then the geometric quotient contains exactly two non-separated points and is obtained from a single point by doubling it asymmetrically. Hence, we can obtain the previous geometric quotient as follows. Consider
 \begin{equation}
 \{(\phi(t,s),w)\in D\times\mathbb C\mid w\not =0\Rightarrow s=t^p \text{ for some }p>0\}
\end{equation}
and take its quotient by $\mathbb C^*$ acting by homotheties on the $\mathbb C$-factor.
\end{remark}

We thus finish with a domain in $\mathbb C^*\times \mathbb C$ non-Hausdorff along a countable set of analytic curves. At each point corresponding to a type IV or a type III Hopf surface, this space is not locally Hausdorff, hence not locally isomorphic to a analytic space.
\begin{remark}
\label{Hausdorfftypes}
Spaces obtained by doubling asymmetrically the points along some subset $C$ are not locally Haudorff along $C$ since every neighborhood of a point $P$ of $C$ contains also the double $Q$ of this point. In particular, any sequence of points converging onto $P$ also converges onto $Q$. This is completely different from the non-Hausdorff spaces obtained as leaf spaces of a foliation with no holonomy (cf. the Teichm\"uller space of simple Hyperk\"ahler manifolds, see \cite{Verbitsky} and Example \ref{HK}). In this last case, given two inseparable points $P$ and $Q$, we can find neighborhoods of $P$ (respectively $Q$) that do not contain $Q$ (respectively $P$). In particular, we can find sequences of points converging to $P$ and not converging to $Q$ (and vice versa). Such spaces are locally Hausdorff.
\end{remark}
\end{example}

\begin{example}{\bf Hirzebruch surfaces.}
\label{Hirzebruchbis}
We go back to the Hirzebruch surfaces of Example \ref{Hirzebruch}. Let $a>0$. To describe $\mathcal M(X,\I (a))$, we see from Figure \ref{Hirzebruchfig} that it is enough to use a single Kuranishi space, that of $\mathbb F_{2a}$. It is equal to $\mathbb C^{2a-1}$ and decomposes as a sequence of algebraic cones (cf. \cite{Cat}, p.21). To be more precise, for any $k\geq 0$, define the algebraic cone
\begin{equation}
\label{Tk}
T_k:=\left\{v\in\mathbb C^{2a-1}\mid \text{rank }
\begin{pmatrix}
v_1 &\hdots &v_{k+1}\cr
\vdots &&\vdots\cr
v_{2a-k-1} &\hdots &v_{2a-1}
\end{pmatrix}
\leq k
\right \}
\end{equation}
of dimension $\min (2a-1, 2k)$. For any $b\leq a$, a point $x$ of $\mathbb C^{2a-1}$ encodes the surface $\mathbb F_{2b}$ if and only if 
\begin{equation}
\label{Hirzcond}
x\in T_{a-b}\setminus T_{a-b-1}. 
\end{equation} 
Taking into account that 
\begin{equation}
\label{hirzconst}
h^0(\mathbb F_{2b})=2b+5\text{ for }b>0\qquad\text\quad h^0(\mathbb F_0)=6
\end{equation}
one may check that $\text{Aut}(\mathbb F_{2a})$ acts on $\mathbb C^{2a-1}$ transitively on each cone (this follows directly from Proposition \ref{autfol}). 
\vspace{5pt}\\
Now, we have to take care of the action of the mapping class group, computed in Lemma \ref{Hirzebruchcc} and Corollary \ref{Hirzebruchcc2}. We can focus on a single connected component of structures, since they are all identified. Hence, we only have to encode the action of the switching map $g$ of \eqref{gaction}. This amounts to consider two copies of 
\begin{equation}
\label{hirz12}
\text{Aut}(\mathbb F_{2a})\times \mathbb C^{2a-1}\rightrightarrows \mathbb C^{2a-1}
\end{equation}
and to add the following morphisms: first a holonomy morphism sending a point
\begin{equation}
\label{hirzglu1}
z\in T_a\setminus T_{a-1}=\mathbb C^{2a-1}\setminus T_{a-1}
\end{equation}
belonging to the first copy of $\mathbb C^{2a-1}$ to the same point in the second copy. Notice that, because of \eqref{Hirzcond}, such a point encodes $\mathbb P^1\times\mathbb P^1$. This holonomy morphism is not defined on the points encoding the other Hirzebruch surfaces. This reflects the fact, explained in Lemma \ref{Hirzebruchcc}, that the set of $\mathbb P^1\times\mathbb P^1$ in a connected component of structures is connected whereas that of the other Hirzebruch surfaces has two connected components. Then we add the action of $g$, which switches the two copies of $\mathbb C^{2a-1}$.
\vspace{5pt}\\
Geometrically, we end with a single copy of $\mathbb C^{2a-1}$, but with two (non-separated) copies of the cone $T_{a-1}$. In other words, adapting the vocabulary of Example \ref{Hopfbis}, we double {\it symmetrically} the points of $\mathbb C^{2a-1}$ along the cone $T_{a-1}$. The automorphism group of $\mathbb F_{2a}$ acts as previously described and the automorphism $g$ of $\mathbb P^1\times\mathbb P^1$ fixes $\mathbb C^{2a-1}$ but exchanges the two copies of the cone.
\vspace{5pt}\\
More formally, the set of objects of $\mathcal M(X,\I(a))$ is
\begin{equation}
\label{Hirzobjects}
\mathbb C^{2a-1}\uniondisjointe \mathbb C^{2a-1}
\end{equation}
and the set of morphisms is generated from
\begin{equation}
\label{Hirzmorphisms}
\begin{aligned}
&\text{Aut}(\mathbb F_{2a})\times \mathbb C^{2a-1}\uniondisjointe \text{Aut}(\mathbb F_{2a})\times \mathbb C^{2a-1}\cr
\uniondisjointe &\text{Aut}(\mathbb F_{2a})\times (\mathbb C^{2a-1}\setminus T_{a-1})\uniondisjointe \text{Aut}(\mathbb F_{2a})\times \mathbb C^{2a-1}.
\end{aligned}
\end{equation}
The third component corresponds to the holonomy morphism \eqref{hirzglu1} and the fourth one to $g$. Source, target and composition can easily be described and we omit the details (cf. the more complicated Example \ref{Hopfbis}).
\vspace{5pt}\\
This describes completely $\mathcal M(X,\I(a))$ but also $\mathcal T(X,\I(a)\cap \I_0)$. In this last case, perform exactly the same construction, but forget about the $g$-identification, that is drop the fourth component of \eqref{Hirzmorphisms}. The geometric quotients of $\mathcal M(X)$ (respectively $\mathcal M(X,\I(a))$) and $\mathcal T(X,\mathcal I_0)$ (respectively $\mathcal T(X,\I(a)\cap \I_0)$) are respectively
\begin{itemize}
\item $\mathbb N$ (respectively $\{0,\hdots,a\}$) with $b\in \mathbb N$ encoding $\mathbb F_{2b}$ and with open sets given by $\{0\}$, $\{0,1\}$, $\{0,1,2\}$ and so on and
\item $\mathbb Z$ (respectively $\{-a,\hdots,a\}$) with $\pm b$ encoding $\mathbb F_{2b}$ and with open sets generated by $\{0\}$, $\{0,1\}$, $\{0,1,2\}$ and so on from the one hand, $\{-1,0\}$, $\{-2,-1,0\}$ and so on from the other hand.
\end{itemize}
\end{example}

\end{document}